\documentclass[12pt]{amsart}       %was 12pt
\usepackage{txfonts}
\usepackage{amssymb}
\usepackage{eucal}
\usepackage{graphicx}
\usepackage{amsmath}
\usepackage{amscd}
\usepackage[all]{xy}           %xypic macro for latex2.09
\usepackage{amsfonts,latexsym}
\usepackage{xspace}
\usepackage{epsfig}
\usepackage{float}
\usepackage{needspace}
\usepackage{color}
\usepackage{fancybox}
\usepackage{multicol}
\usepackage{colordvi}

\usepackage{ifpdf}
\ifpdf
\usepackage[colorlinks,final,backref=page,hyperindex]{hyperref}
\else
\usepackage[colorlinks,final,backref=page,hyperindex,hypertex]{hyperref}
\fi

\usepackage[active]{srcltx} %SRC Specials for DVI Searching
\usepackage{tikz}

\usepackage{mathrsfs}
\newtheorem{theorem}{Theorem}[section]
\newtheorem{prop}[theorem]{Proposition}

\newtheorem{coro}[theorem]{Corollary}
\newtheorem{prop-def}{Proposition-Definition}[section]
\newtheorem{coro-def}{Corollary-Definition}[section]

\theoremstyle{definition}
\newtheorem{defn}[theorem]{Definition}
\newtheorem{remark}[theorem]{Remark}

\newtheorem{exam}[theorem]{Example}

\newcommand{\nc}{\newcommand}
\nc{\tred}[1]{\textcolor{red}{#1}}
\nc{\tblue}[1]{\textcolor{blue}{#1}}
\nc{\tgreen}[1]{\textcolor{green}{#1}}
\nc{\tpurple}[1]{\textcolor{purple}{#1}}
\nc{\btred}[1]{\textcolor{red}{\bf #1}}
\nc{\btblue}[1]{\textcolor{blue}{\bf #1}}
\nc{\btgreen}[1]{\textcolor{green}{\bf #1}}
\nc{\btpurple}[1]{\textcolor{purple}{\bf #1}}
\nc{\NN}{{\mathbb N}}
\nc{\ncsha}{{\mbox{\cyr X}^{\mathrm NC}}} \nc{\ncshao}{{\mbox{\cyr
X}^{\mathrm NC}_0}}

\newcommand{\efootnote}[1]{}
\newcommand{\delete}[1]{}

\delete{
\nc{\mlabel}[1]{\label{#1}}  % Use this to suppress names
\nc{\mcite}[1]{\cite{#1}}  % Use this to suppress names
\nc{\mref}[1]{\ref{#1}}  % Use this to suppress names
\nc{\mbibitem}[1]{\bibitem{#1}} % Use this to show number
}

\nc{\mlabel}[1]{\label{#1}{\hfill \hspace{1cm}{\bf{{\ }\hfill(#1)}}}}
\nc{\mcite}[1]{\cite{#1}{{\bf{{\ }(#1)}}}}  % Use this lines to show names
\nc{\mref}[1]{\ref{#1}{{\bf{{\ }(#1)}}}}  % Use this lines to show names
\nc{\mbibitem}[1]{\bibitem[\bf #1]{#1}} % Use this to show name
\nc{\opa}{\ast} \nc{\opb}{\odot} \nc{\op}{\bullet} \nc{\pa}{\frakL}
\nc{\arr}{\rightarrow} \nc{\lu}[1]{(#1)} \nc{\mult}{\mrm{mult}}
\nc{\diff}{\mathfrak{Diff}}
\nc{\opc}{\sharp}\nc{\opd}{\natural}
\nc{\ope}{\circ}
\nc{\dpt}{\mathrm{d}}
\nc{\hck}{H_{RT}}
\nc{\vdf}{\calf}
\nc{\ldf}{\calf_\ell}
\nc{\hlf}{H_\ell}
\nc{\onek}{\mathbf{1}_\bfk}
\nc{\tforall}{\quad \text{ for all }}

\nc{\AW}{\mathcal{A}}

\nc{\rba}{Rota-Baxter algebra\xspace}
\nc{\rbas}{Rota-Baxter algebras\xspace}

\nc{\ari}{\mathrm{ar}}

\nc{\lef}{\mathrm{lef}}

\nc{\Sh}{\mathrm{ST}}

\nc{\Cr}{\mathrm{Cr}}

\nc{\st}{{Schr\"oder tree}\xspace}
\nc{\sts}{{Schr\"oder trees}\xspace}

\nc{\vertset}{\Omega} % set of vertex decorations

\nc{\assop}{\quad \begin{picture}(5,5)(0,0)
\line(-1,1){10}
\put(-2.2,-2.2){$\bullet$}
\line(0,-1){10}\line(1,1){10}
\end{picture} \quad \smallskip}

\nc{\operator}{\begin{picture}(5,5)(0,0)
\line(0,-1){6}
\put(-2.6,-1.8){$\bullet$}
\line(0,1){9}
\end{picture}}

\nc{\idx}{\begin{picture}(6,6)(-3,-3)
\put(0,0){\line(0,1){6}}
\put(0,0){\line(0,-1){6}}
\end{picture}}

\nc{\pb}{{\mathrm{pb}}}
\nc{\Lf}{{\mathrm{Lf}}}

\nc{\lft}{{left tree}\xspace}
\nc{\lfts}{{left trees}\xspace}

\nc{\fat}{{fundamental averaging tree}\xspace}

\nc{\fats}{{fundamental averaging trees}\xspace}
\nc{\avt}{\mathrm{Avt}}

\nc{\rass}{{\mathit{RAss}}}

\nc{\aass}{{\mathit{AAss}}}

\nc{\vin}{{\mathrm Vin}}    %decoration set of indices
\nc{\lin}{{\mathrm Lin}}    %decoration set of leaves
\nc{\inv}{\mathrm{I}n}
\nc{\gensp}{V} % space of generators
\nc{\genbas}{\mathcal{V}} % basis of the space of generators
\nc{\bvp}{V_P}     % Rota-Baxter generating space
\nc{\gop}{{\,\omega\,}}     % generic binary operation

\nc{\bin}[2]{ (_{\stackrel{\scs{#1}}{\scs{#2}}})}  %binomial coeff
\nc{\binc}[2]{ \left (\!\! \begin{array}{c} \scs{#1}\\
    \scs{#2} \end{array}\!\! \right )}  %binomial coeff
\nc{\bincc}[2]{  \left ( {\scs{#1} \atop
    \vspace{-1cm}\scs{#2}} \right )}  %binomial coeff
\nc{\bs}{\bar{S}} \nc{\cosum}{\sqsubset} \nc{\la}{\longrightarrow}
\nc{\rar}{\rightarrow} \nc{\dar}{\downarrow} \nc{\dprod}{**}
\nc{\dap}[1]{\downarrow \rlap{$\scriptstyle{#1}$}}
\nc{\md}{\mathrm{dth}} \nc{\uap}[1]{\uparrow
\rlap{$\scriptstyle{#1}$}} \nc{\defeq}{\stackrel{\rm def}{=}}
\nc{\disp}[1]{\displaystyle{#1}} \nc{\dotcup}{\
\displaystyle{\bigcup^\bullet}\ } \nc{\gzeta}{\bar{\zeta}}
\nc{\hcm}{\ \hat{,}\ } \nc{\hts}{\hat{\otimes}}
\nc{\barot}{{\otimes}} \nc{\free}[1]{\bar{#1}}
\nc{\uni}[1]{\tilde{#1}} \nc{\hcirc}{\hat{\circ}} \nc{\lleft}{[}
\nc{\lright}{]} \nc{\lc}{\lfloor} \nc{\rc}{\rfloor}
\nc{\curlyl}{\left \{ \begin{array}{c} {} \\ {} \end{array}
    \right .  \!\!\!\!\!\!\!}
\nc{\curlyr}{ \!\!\!\!\!\!\!
    \left . \begin{array}{c} {} \\ {} \end{array}
    \right \} }
\nc{\longmid}{\left | \begin{array}{c} {} \\ {} \end{array}
    \right . \!\!\!\!\!\!\!}
\nc{\onetree}{\bullet} \nc{\ora}[1]{\stackrel{#1}{\rar}}
\nc{\ola}[1]{\stackrel{#1}{\la}}%${\Bbb Z}$
\nc{\ot}{\otimes} \nc{\mot}{{{\boxtimes\,}}}
\nc{\otm}{\overline{\boxtimes}} \nc{\sprod}{\bullet}
\nc{\scs}[1]{\scriptstyle{#1}} \nc{\mrm}[1]{{\rm #1}}
\nc{\margin}[1]{\marginpar{\rm #1}}   %{\rm #1}}
\nc{\dirlim}{\displaystyle{\lim_{\longrightarrow}}\,}
\nc{\invlim}{\displaystyle{\lim_{\longleftarrow}}\,}
\nc{\mvp}{\vspace{0.3cm}} \nc{\tk}{^{(k)}} \nc{\tp}{^\prime}
\nc{\ttp}{^{\prime\prime}} \nc{\svp}{\vspace{2cm}}
\nc{\vp}{\vspace{8cm}} \nc{\proofbegin}{\noindent{\bf Proof: }}
\nc{\proofend}{$\blacksquare$ \vspace{0.3cm}}
\nc{\modg}[1]{\!<\!\!{#1}\!\!>}
\nc{\intg}[1]{F_C(#1)} \nc{\lmodg}{\!
<\!\!} \nc{\rmodg}{\!\!>\!}
\nc{\cpi}{\widehat{\Pi}}
\nc{\sha}{{\mbox{\cyr X}}}  %used to be \cyr

\newfont{\scyr}{wncyr10 scaled 550}
\nc{\ssha}{\mbox{\bf \scyr X}}
\nc{\shap}{{\mbox{\cyrs X}}} %sha as product
\nc{\shpr}{\diamond}    %Shuffle product
\nc{\shp}{\ast} \nc{\shplus}{\shpr^+}
\nc{\shprc}{\shpr_c}    %Cartier's product
\nc{\msh}{\ast} \nc{\zprod}{m_0} \nc{\oprod}{m_1}
\nc{\vep}{\epsilon} \nc{\labs}{\mid\!} \nc{\rabs}{\!\mid}
\nc{\sqmon}[1]{\langle #1\rangle}
\nc{\mmbox}[1]{\mbox{\ #1\ }} \nc{\dep}{\mrm{dep}} \nc{\fp}{\mrm{FP}}
\nc{\rchar}{\mrm{char}} \nc{\End}{\mrm{End}} \nc{\Fil}{\mrm{Fil}}
\nc{\Mor}{Mor\xspace} \nc{\gmzvs}{gMZV\xspace}
\nc{\gmzv}{gMZV\xspace} \nc{\mzv}{MZV\xspace}
\nc{\mzvs}{MZVs\xspace} \nc{\Hom}{\mrm{Hom}} \nc{\id}{\mrm{id}}
\nc{\im}{\mrm{im}} \nc{\incl}{\mrm{incl}} \nc{\map}{\mrm{Map}}
\nc{\mchar}{\rm char} \nc{\nz}{\rm NZ} \nc{\supp}{\mathrm Supp}

\nc{\Alg}{\mathbf{Alg}} \nc{\Bax}{\mathbf{Bax}} \nc{\bff}{\mathbf f}
\nc{\bfk}{{\bf k}} \nc{\bfone}{{\bf 1}} \nc{\bfx}{\mathbf x}
\nc{\bfy}{\mathbf y}
\nc{\base}[1]{\bfone^{\otimes ({#1}+1)}} %{{a_{#1}}}
\nc{\Cat}{\mathbf{Cat}}

\nc{\detail}{\marginpar{\bf More detail}
    \noindent{\bf Need more detail!}
    \svp}
\nc{\Int}{\mathbf{Int}} \nc{\Mon}{\mathbf{Mon}}
\nc{\rbtm}{{shuffle }} \nc{\rbto}{{Rota-Baxter }}
\nc{\remarks}{\noindent{\bf Remarks: }} \nc{\Rings}{\mathbf{Rings}}
\nc{\Sets}{\mathbf{Sets}} \nc{\wtot}{\widetilde{\odot}}
\nc{\wast}{\widetilde{\ast}} \nc{\bodot}{\bar{\odot}}
\nc{\bast}{\bar{\ast}} \nc{\hodot}[1]{\odot^{#1}}
\nc{\hast}[1]{\ast^{#1}} \nc{\mal}{\mathcal{O}}
\nc{\tet}{\tilde{\ast}} \nc{\teot}{\tilde{\odot}}
\nc{\oex}{\overline{x}} \nc{\oey}{\overline{y}}
\nc{\oez}{\overline{z}} \nc{\oef}{\overline{f}}
\nc{\oea}{\overline{a}} \nc{\oeb}{\overline{b}}
\nc{\weast}[1]{\widetilde{\ast}^{#1}}
\nc{\weodot}[1]{\widetilde{\odot}^{#1}} \nc{\hstar}[1]{\star^{#1}}
\nc{\lae}{\langle} \nc{\rae}{\rangle}
\nc{\lf}{\lfloor}
\nc{\rf}{\rfloor}
\nc{\QQ}{{\mathbb Q}}
\nc{\RR}{{\mathbb R}} \nc{\ZZ}{{\mathbb Z}}

\nc{\cala}{{\mathcal A}} \nc{\calb}{{\mathcal B}}
\nc{\calc}{{\mathcal C}}
\nc{\cald}{{\mathcal D}} \nc{\cale}{{\mathcal E}}
\nc{\calf}{{\mathcal F}} \nc{\calg}{{\mathcal G}}
\nc{\calh}{{\mathcal H}} \nc{\cali}{{\mathcal I}}
\nc{\call}{{\mathcal L}} \nc{\calm}{{\mathcal M}}
\nc{\caln}{{\mathcal N}} \nc{\calo}{{\mathcal O}}
\nc{\calp}{{\mathcal P}} \nc{\calr}{{\mathcal R}}
\nc{\cals}{{\mathcal S}} \nc{\calt}{{\mathcal T}}
\nc{\calu}{{\mathcal U}} \nc{\calw}{{\mathcal W}} \nc{\calk}{{\mathcal K}}
\nc{\calx}{{\mathcal X}} \nc{\CA}{\mathcal{A}}

\nc{\fraka}{{\mathfrak a}} \nc{\frakA}{{\mathfrak A}}
\nc{\frakb}{{\mathfrak b}} \nc{\frakB}{{\mathfrak B}}
\nc{\frakc}{{\mathfrak c}}
\nc{\frakD}{{\mathfrak D}} \nc{\frakF}{\mathfrak{F}}
\nc{\frakf}{{\mathfrak f}} \nc{\frakg}{{\mathfrak g}}
\nc{\frakH}{{\mathfrak H}} \nc{\frakL}{{\mathfrak L}}
\nc{\frakM}{{\mathfrak M}} \nc{\bfrakM}{\overline{\frakM}}
\nc{\frakm}{{\mathfrak m}} \nc{\frakP}{{\mathfrak P}}
\nc{\frakN}{{\mathfrak N}} \nc{\frakp}{{\mathfrak p}}
\nc{\frakS}{{\mathfrak S}} \nc{\frakT}{\mathfrak{T}}
\nc{\frakX}{{\mathfrak X}}
\nc{\BS}{\mathbb{S
}}

\font\cyr=wncyr10 \font\cyrs=wncyr7
\nc{\ID}{{\rm I}}\nc{\lbar}[1]{\overline{#1}}\nc{\bre}{{\rm bre}}
\nc{\sd}{\cals}\nc{\rb}{\rm RB}\nc{\A}{\rm A}\nc{\LL}{\rm L}\nc{\tx}{\tilde{X}}
\nc{\col}{\Delta_{RT}}\nc{\mul}{m_{RT}}\nc{\ul}{u_{RT}}\nc{\epl}{\varepsilon_{RT}}
\nc{\hl}{H_{RT}}\nc{\arro}[1]{#1}\nc{\px}{P_{\tx}}\nc{\pw}{P_{\mathfrak{w}}}\nc{\pl}{B_\omega^+}
\nc{\pp}{\pl}\nc{\ppp}[1]{B^+(#1)}\nc{\dw}{\diamond_{\mathfrak{w}}}\nc{\dl}{\diamond_{\rm \ell}}
\nc{\ncshaw}{\sha^{{\rm NC}}_{\Omega}}\nc{\ncshal}{\sha^{{\rm NC}}_{{\rm RT}}}
\nc{\ver}{\rm V}\nc{\ld}{l}\nc{\del}{\Delta_{{\rm \ell}}}\nc{\epsl}{\epsilon_{{\rm \ell}}}
\nc{\uul}{u_{{\rm \ell}}}\nc{\oneh}{\mathbf{1}}\nc{\onew}{\mathbf{1}}
\nc{\etree}{1} \nc{\conc}{m_{RT}}
\nc{\hrtb}{H_{RT}(X\sqcup\Omega)} \nc{\hrts}{H_{RT}(X, \Omega)}\nc{\rts}{\mathcal{T}(X, \Omega)}\nc{\rfs}{\mathcal{F}(X, \Omega)} \nc{\ncshall}{\sha^{{\rm NC}}_{{\rm RT}}} \nc{\ldl}{\leq_{\mathrm{dl}}} \nc{\pla}{B_{\alpha}^{+}} \nc{\plb}{B_{\beta}^{+}}

\nc{\bim}[1]{#1}  \nc{\shaop}{\sha_{\Omega}^{+}}  \nc{\shao}{\sha_{\Omega}}
\nc{\bbim}[2]{#1 #2} \nc{\bbbim}[2]{#1,\, #2} \nc{\RBF}{{\rm RBF}}
\nc{\frbf}{F_{\RBF}} \nc{\shaf}{\ssha_{\tiny{\Omega}}} \nc{\sham}{\diamond_{\Omega}}
\nc{\dnx}{\Delta_n A} \nc{\dx}{\Delta A} \nc{\dgp}{{\rm deg_{P}}}
\nc{\dgt}{{\rm deg_{T}}} \nc{\dg}{{\rm deg}} \nc{\ida}{ID($A$)} \nc{\tu}{\tilde{u}} \nc{\tv}{\tilde{v}}
 \nc{\fbase}{\calb} \nc{\LF}{\mathrm{RF}} \nc{\FFA}{\mathrm{LF}} \nc{\irr}{\mathrm{Irr}}
 \nc{\result}{\bfk\mathrm{Irr}(S_n)}  \nc{\I}{I_{\mathrm{ID},n}^0}
 \nc{\nrs}{\calr_n^\star} \nc{\ii}{\mathrm{I}} \nc{\iii}{\mathrm{II}}
\nc{\intl}{{\rm int}}\nc{\ws}[1]{{#1}}\nc{\deleted}[1]{\delete{#1}}\nc{\plas}{placements\xspace}

\nc{\Id}{\mathrm{Id}} \nc{\Irr}{\mathrm{Irr}}

\nc{\tos}{totally ordered set } \nc{\nes}{nonempty set}
\nc{\baa}{\succ_\alpha}
\nc{\laa}{\prec_\alpha}
\nc{\bba}{\succ_\beta}
\nc{\lba}{\prec_\beta}
\nc{\bdt} {\succ_\delta}
\nc{\lia} {\prec_\iota}
\nc{\bka}{\succ_\kappa}
\nc{\lmu}{\prec_\mu}
\nc{\Po}{(P_\omega)_{\omega\in \Omega}}

\nc{\apr}{\cdot} \nc{\spr}{\cdot}

\usetikzlibrary{calc}
\newlength\xch\newlength\dbj
\newif\ifqdd
\begin{document}

\title[Monoidal structure on Rota-Baxter algebra]{Monoidal structure on Rota-Baxter algebras via augmentation, and free cocycle bialgebras}
%{A monoidal categorical approach to augmented Rota-Baxter algebras and free cocycle bialgebras}
\author{Li Guo}
\address{Department of Mathematics and Computer Science, Rutgers University, Newark, NJ 07102, US}
\email{liguo@rutgers.edu}

\author{Yunzhou Xie}\address{Department of Mathematics, Imperial College London, SW7 2AZ, London, UK}
\email{yx3021@ic.ac.uk}

\author{Yi Zhang}
\address{School of Mathematics and Statistics, Center for Applied Mathematics of Jiangsu Province,
Nanjing University of Information Science and Technology, Nanjing, Jiangsu, 210044, P.\,R. China}
\email{zhangy2016@nuist.edu.cn}

\begin{abstract}
A tensor product is known for commutative Rota--Baxter algebras through a universal construction, but its underlying algebra need not be the ordinary tensor product $A\otimes B$ and can be substantially larger.  We show that, after passing to augmented Rota--Baxter algebras, the ordinary tensor product itself carries a natural Rota--Baxter operator
which makes the category monoidal with the usual associator and unit constraints.  The resulting structure is intrinsically nonsymmetric: over a field of characteristic zero, this operator and its flip-conjugated counterpart are the only possibilities among linear combinations of five natural terms, while over any field the full category admits no braiding.  Its comonoid objects are precisely Rota--Baxter cocycle bialgebras.  Moreover, the free Rota--Baxter adjunction yields a bimonad and free cocycle bialgebras on ordinary bialgebras, with compatible Hopf structures for connected filtered inputs, in both the associative and commutative settings.
\end{abstract}

\subjclass[2020]{18M05, 16T05, 16T10, 17B38, 16W99}

\keywords{augmented Rota-Baxter algebra, monoidal category, tensor product, bimonad, cocycle bialgebra, Hopf algebra}

\maketitle

\setcounter{tocdepth}{2}

\vspace{-1cm}

\tableofcontents

\vspace{-1cm}

\section{Introduction}

The purpose of this paper is to construct and study a monoidal structure on augmented Rota--Baxter algebras whose underlying algebraic product is the ordinary tensor product. The basic problem is that the ordinary tensor product of two Rota--Baxter operators does not, in general, define a Rota--Baxter operator on the tensor product algebra. We show that augmentations supply a natural replacement and lead to a monoidal category whose structure is intrinsically oriented: within a natural five-term class there are only two flip-conjugate possibilities, and the full category admits no braiding. This asymmetry is precisely what turns the Hochschild $1$-cocycle identity into the condition that a coproduct be a morphism in the monoidal category. Consequently, Rota--Baxter cocycle bialgebras become its comonoid objects, and the free Rota--Baxter adjunction produces a bimonad and free cocycle bialgebras from ordinary bialgebras.

\subsection{Background and motivation}

A Rota-Baxter algebra of weight $\lambda$~\cite{B60,R69} is an associative
algebra $A$ equipped with a linear operator $P:A\to A$ satisfying
$$
P(x)P(y)=P\bigl(P(x)y+xP(y)+\lambda xy\bigr),
\qquad x,y\in A.
$$
For $\lambda=0$, this identity abstracts integration by parts.
For nonzero weights, it includes algebraic forms of summation
by parts and appears in combinatorics, operads, Hopf algebras, quantum field theory and integrable systems. 
\cite{Ag01,AGKO03,Bai,BBGN,BG,BD,CK00,G12,GK00a,LS,Ro2,STS,WZ,YGT}.

There have been numerous categorical studies of Rota-Baxter algebras in several directions, including Rota-Baxter categories~\cite{CD09}, Rota-Baxter algebra objects in braided or symmetric monoidal categories~\cite{Ag20,GL21,Ji13,LG22}, and categorical and multicategorical formulations of the Rota-Baxter identity~\cite{S16}, as well as approaches via monads and distributive laws, twisting constructions, and species~\cite{FGPXZ,PVO17,ZGK16}. 

In the present paper, we are interested specifically in equipping the category of Rota-Baxter algebras itself with a monoidal structure.
For commutative Rota-Baxter algebras, the Rota-Baxter tensor product constructed in~\cite{CG14} provides such a categorical structure through its universal property. Its underlying algebra, however, need not be the ordinary tensor product $A\otimes B$ and can be substantially larger. Thus a natural question arises:
\begin{quote}
\itshape
Is it possible to equip a category of Rota-Baxter algebras with a monoidal structure while retaining the ordinary tensor product as the underlying algebra?
\end{quote}
This requires a Rota-Baxter operator on the ordinary tensor product $A\otimes B$ that is functorial in the two factors and compatible with the usual associator and unit constraints. We address this question in the present paper.

\subsection{The augmented structure and the tensor-product problem}

Let $(A,P_A)$ and $(B,P_B)$ be Rota-Baxter algebras of the same weight. The operator $P_A\otimes P_B$ generally fails to satisfy the Rota-Baxter identity on $A\otimes B$. We resolve this difficulty by passing to augmented Rota-Baxter algebras. 

An augmented Rota-Baxter algebra is a triple
$(A,P_A,\varepsilon_A)$, where $(A,P_A)$ is a unital
Rota-Baxter algebra and $\varepsilon_A:A\to\mathbf{k}$
is a unital algebra morphism satisfying $\varepsilon_AP_A=0$.
The augmentation gives a direct sum of $\mathbf{k}$-modules
$A=\mathbf{k}1_A\oplus I_A$, where $I_A=\ker\varepsilon_A$.
Since $P_A(A)\subseteq I_A$, the restriction $Q_A=P_A|_{I_A}$
and the element $p_A=P_A(1_A)$ determine $P_A$.
Adjoining a unit to a nonunital Rota-Baxter algebra is more
delicate than adjoining a unit to an associative algebra
\cite{G12}.
In the augmented setting, the Rota-Baxter identity gives
explicit necessary and sufficient conditions on $(I_A,Q_A,p_A)$.

For an augmented algebra, let
$u_A:\mathbf{k}\to A$ be its unit map and put
$e_A=u_A\varepsilon_A$.
For two augmented Rota-Baxter algebras of weight $\lambda$,
define
$$
P_{A\boxtimes B}=P_A\otimes e_B+\mathrm{id}_A\otimes P_B.
$$
The multiplicativity of $e_B$, its scalar image and
$e_BP_B=0$ allow the mixed terms in the Rota-Baxter identity
to be combined.
The ordinary associator and the unit object
$(\mathbf{k},0,\mathrm{id}_{\mathbf{k}})$ then give a monoidal
category.
The formula distinguishes the two factors: the term involving
$P_A$ projects the second factor onto its scalar part, whereas
the term involving $P_B$ leaves the first factor unchanged.

This orientation is an essential structural feature rather than an artifact of the formula. Over a field of characteristic zero, if one searches among linear combinations of the five natural terms
\[
P_A\otimes\id_B,\quad P_A\otimes e_B,\quad \id_A\otimes P_B,\quad e_A\otimes P_B,\quad P_A\otimes P_B,
\]
compatibility with the ordinary associator and unit constraints leaves only the above operator and its flip-conjugated counterpart. Moreover, over any field the resulting full monoidal category admits no braiding. Thus the order dependence cannot be removed by replacing the ordinary flip with another family of commutativity isomorphisms.

The same oriented formula also explains the cocycle condition often appearing in Hopf algebras.
For a bialgebra $H$ with counit $\varepsilon_H$,
the counit identities turn
$(P\otimes e_H)\Delta_H$ into the map
$h\mapsto P(h)\otimes1_H$.
Consequently, requiring $\Delta_H$ to intertwine $P$ with the Rota-Baxter operator $P_{H\boxtimes H}$ gives the Hochschild $1$-cocycle identity.
This relates the construction to cocycle Hopf algebras on rooted trees and free Rota-Baxter algebras
\cite{CK00,ZGG16}.

A direct predecessor in the free commutative setting is the construction for
extended Rota-Baxter algebras in \cite{ZGQ24}.
Setting $\kappa=\mu=0$ in their Eq.~(25) gives the operator
$P\otimes e+\mathrm{id}\otimes P$ on the tensor
square of a free commutative Rota-Baxter algebra.
Their bialgebra and connected filtered Hopf constructions then
give the corresponding free commutative cases considered here.
The present work develops a different level of structure: the operator is defined for arbitrary pairs of augmented Rota--Baxter algebras, including noncommutative ones; its coherence and rigidity are analyzed; the failure of braidings is proved; and the cocycle condition is characterized categorically through comonoid objects. This categorical structure in turn makes the free and forgetful adjunction the mechanism that produces the free cocycle bialgebra constructions.

\subsection{An outline with main results}
First in Section~\ref{sec:structure}, we establish the structural properties of augmented Rota--Baxter algebras needed for the monoidal construction. An augmented Rota--Baxter algebra can be reconstructed from its augmentation ideal, the restricted operator, and the value of the operator at the unit. We give three necessary and sufficient identities for a Rota--Baxter operator on the ideal to extend to the ordinary unitization (Theorem~ \ref{thm:kernel-reconstruction}). When the operator vanishes at the unit, we call the augmented Rota--Baxter algebra normalized; in this case the extension conditions reduce to a quadratic relation on the restricted operator, leading to a concrete description of normalized objects (Theorem~\ref{thm:unit-zero}).

The central construction appears in Section~\ref{sec:tp}. The operator
\[
P_{A\boxtimes B}=P_A\otimes e_B+\id_A\otimes P_B
\]
makes the ordinary tensor product $A\otimes B$ into an augmented Rota--Baxter algebra and is compatible with morphisms, the ordinary associator, and the ordinary unit constraints (Theorems~\ref{thm:tensor-product} and \ref{thm:monoidal}). Thus augmented Rota--Baxter algebras form a monoidal category with the usual tensor product on underlying algebras. The construction is rigid in two complementary senses. Over a field of characteristic zero, it and its flip-conjugated counterpart are the only monoidal products arising from a specified five-term class of natural operator formulas with the ordinary constraints (Theorem~\ref{thm:tensor-uniqueness}). More generally, the asymmetry cannot be removed by choosing a different braiding: over any field the full monoidal category admits no braiding (Theorem~\ref{thm:not-braided}). In this sense the monoidal structure is intrinsically oriented.

Section~\ref{sec:monoidal} shows that this orientation has a direct coalgebraic meaning. A coproduct $\Delta:H\to H\boxtimes H$ is a morphism of augmented Rota--Baxter algebras exactly when it satisfies the Hochschild $1$-cocycle identity. Consequently, the comonoid objects are precisely Rota--Baxter cocycle bialgebras (Theorem~\ref{thm:comonoid}). The forgetful functor to augmented algebras is strict monoidal, so its left adjoint, the free Rota--Baxter functor, is opmonoidal. The induced monad is a bimonad, and its Eilenberg--Moore monoidal product recovers the concrete product $\boxtimes$ (Theorem~\ref{thm:bimonad}). Thus the categorical formulation does not merely repackage the tensor formula: it reconstructs the same monoidal structure from the free Rota--Baxter adjunction and explains why free objects transport ordinary bialgebra structures to cocycle bialgebra structures.

This is applied in Section~\ref{sec:frb}. For a bialgebra $A$ with its prescribed multiplication and coproduct, freeness extends $\Delta_A$ uniquely to a cocycle coproduct on the free Rota--Baxter algebra $F_{\lambda}(A)$. This construction is left adjoint to the forgetful functor from Rota--Baxter cocycle bialgebras to bialgebras (Theorem~\ref{thm:free-lifting}). The commutative analogue is obtained from the free commutative adjunction (Corollary~\ref{cor:free-commutative}). This differs from free constructions generated only by sets \cite{ZGG16,ZXG22}: the input here is already a bialgebra, and its multiplication and coproduct are retained and extended through the monoidal universal property.

For connected filtered inputs, both the associative and commutative free constructions carry compatible connected filtrations and hence admit antipodes (Theorem~\ref{thm:filtered-hopf}). Finally, the iterates $P^n(1)$ connect the classical $\lambda$-divided-power and Stirling-number multiplication identities of \cite{AGKO03,Guo05} with a different coproduct forced by the cocycle identity. This yields explicit coproduct and antipode formulas and illustrates concretely how the new monoidal viewpoint changes the coalgebraic structure while retaining the familiar Rota--Baxter multiplication formulas.

\smallskip

\noindent
\textbf{Notations.}
Throughout, $\mathbf{k}$ is a commutative ring with unit,
$\lambda\in\mathbf{k}$ is fixed, and tensor products are taken
over $\mathbf{k}$.
Algebras are associative and unital unless stated otherwise.
For a unital algebra $A$, write $u_A:\mathbf{k}\to A$ for
its unit map and $1_A=u_A(1)$.
For an augmented algebra $(A,\varepsilon_A)$, put
$I_A=\ker\varepsilon_A$ and $e_A=u_A\varepsilon_A$.
Thus $e_A$ is the algebra projection onto $\mathbf{k}1_A$
along $I_A$.
The extra field, characteristic and invertibility assumptions
are stated where needed.

\section{Augmented Rota-Baxter algebras}\label{sec:structure}
In this section, we introduce augmented Rota-Baxter algebras and record the structural properties needed later. In particular, we describe them in terms of the augmentation ideal, the restricted operator, and the value of the operator at the unit, and then specialize to the normalized case.

\subsection{Definitions and examples}
We begin with basis notions and examples. 

\begin{defn}[\cite{G12}]\label{def:rb}
A \textbf{Rota-Baxter algebra of weight $\lambda\in\mathbf{k}$} is a pair $(A,P)$,
where $A$ is a unital associative $\mathbf{k}$-algebra and $P:A\to A$ is a
$\mathbf{k}$-linear map satisfying
$$
P(x)P(y)=P(P(x)y+xP(y)+\lambda xy),
\qquad x,y\in A.
$$
A morphism of Rota-Baxter algebras is a unital algebra morphism that
commutes with the operators. If the underlying algebras are not required
to have units, we speak of \textbf{nonunital Rota-Baxter algebras} and
do not require their morphisms to preserve units.
\end{defn}

\begin{defn}
An \textbf{augmented Rota-Baxter algebra of weight $\lambda$} is a triple
$(A,P,\varepsilon)$, where $(A,P)$ is a Rota-Baxter algebra of weight
$\lambda$ and $\varepsilon:A\longrightarrow\mathbf{k}$
is a unital algebra morphism such that $\varepsilon P=0$. Equivalently,
$\varepsilon:(A,P)\to(\mathbf{k},0)$ is a morphism of Rota-Baxter algebras.
\end{defn}

An augmentation is a character of the underlying algebra. This is the
notion of character used in combinatorial Hopf algebras \cite{ABS06},
although no coalgebra structure is required in the present definition.
The additional condition $\varepsilon P=0$ is essential for the tensor
product constructed in Section \ref{sec:tp}.

\begin{defn}
Let $(A,P_A,\varepsilon_A)$ and $(B,P_B,\varepsilon_B)$ be augmented
Rota-Baxter algebras of the same weight. A \textbf{morphism}
$f:(A,P_A,\varepsilon_A)\to(B,P_B,\varepsilon_B)$ is a unital algebra
morphism satisfying $fP_A=P_Bf$ and
$\varepsilon_Bf=\varepsilon_A$.
The resulting category is denoted by $\mathbf{AugRBAlg}_\lambda$.
\end{defn}

\begin{exam}\label{ex:basic}

\begin{enumerate}
\item\label{it:ex-zero} Let $A$ be a unital augmented algebra. Then
$(A,0,\varepsilon_A)$ is an augmented Rota-Baxter algebra of every
weight and its operator vanishes at $1_A$.

\item\label{it:ex-can} Let $A$ be a unital augmented algebra. The decomposition
$A=\mathbf{k}1_A\oplus I_A$ is a direct sum of
$\mathbf{k}$-modules, and both summands are closed under
multiplication. The
projection $\mathrm{id}_A-e_A$ onto $I_A$ is a Rota-Baxter
operator of weight $-1$ \cite{G12,GL21}. Hence
$P_A=-\lambda(\mathrm{id}_A-e_A)$ defines an augmented Rota-Baxter
algebra of weight $\lambda$ and satisfies
$P_A(1_A)=0$. This operator is used in the
strong monoidal embedding of Proposition \ref{thm:canonical-embedding} and
in the cocycle construction of Corollary
\ref{coro:canonical-cocycle}.

\item\label{it:ex-dual} Let $A=\mathbf{k}[t]/(t^2)$ and write $\bar t$ for the residue
class of $t$. With augmentation $\varepsilon(a+b\bar t)=a$, the operator
$P(a+b\bar t)=-\lambda b\bar t$ is the operator in the preceding
example, namely Example \ref{ex:basic}(\ref{it:ex-can}).

\item\label{it:ex-int} Assume that $\mathbf{k}$ is a field of characteristic zero. On
$A=\mathbf{k}[x]$, let
$$
P(f)(x)=\int_0^x f(t)\,dt,
\qquad
\varepsilon(f)=f(0).
$$
Then $(A,P,\varepsilon)$ is an augmented Rota-Baxter algebra of weight
zero. Its value at the unit is $P(1_A)=x$, so this operator does not vanish at the unit.

\item\label{it:ex-sum} Let $A$ be the algebra of $\mathbf{k}$-valued functions on
$\mathbb N=\{0,1,2,\ldots\}$ with pointwise multiplication. Define
$$
P(f)(n)=\sum_{i=0}^{n-1}f(i),
\qquad
\varepsilon(f)=f(0).
$$
The sum is zero when $n=0$.
Then $(A,P,\varepsilon)$ is an augmented Rota-Baxter algebra of weight
$1$, and $P(1_A)$ is the function $n\mapsto n$.
\end{enumerate}
\end{exam}

\subsection{The augmentation ideal and the value at the unit}

Let $(A,P_A,\varepsilon_A)$ be an augmented Rota-Baxter algebra. We use the
notation
$$
p_A:=P_A(1_A),
\qquad
Q_A:=P_A|_{I_A}.
$$
The element $p_A$ records the failure of the unit morphism
$u_A:(\mathbf{k},0)\to(A,P_A)$ to be a Rota-Baxter morphism.

\begin{prop}\label{prop:unit-identities}
Let $(A,P_A,\varepsilon_A)$ be an augmented Rota-Baxter algebra of weight
$\lambda$. Then:
\begin{enumerate}
\item\label{it:a}
$$
A=\mathbf{k}1_A\oplus I_A,
\qquad
P_A(A)\subseteq I_A,
\qquad
P_A(I_A)\subseteq I_A.
$$
The restriction $Q_A:I_A\to I_A$ is a Rota-Baxter operator of weight
$\lambda$ on the nonunital algebra $I_A$.

\item\label{it:b} The image $P_A(A)$ is a subalgebra of $I_A$, not necessarily unital.

\item\label{it:c} We have $p_A\in I_A$ and $p_A^2=2P_A(p_A)+\lambda p_A$.
Moreover, for every $a\in A$,
\begin{align*}
P_A^2(a)+\lambda P_A(a)
&=p_AP_A(a)-P_A(p_Aa)\\
&=P_A(a)p_A-P_A(ap_A).
\end{align*}
\end{enumerate}
\end{prop}

\begin{proof}
(\ref{it:a}) Since $\varepsilon_A$ is unital, for every $a\in A$ we have $a=\varepsilon_A(a)1_A+ \bigl(a-\varepsilon_A(a)1_A\bigr)$,
and the second summand belongs to $I_A$. If $c1_A\in I_A$, then
$c=\varepsilon_A(c1_A)=0$. Hence $A=\mathbf{k}1_A\oplus I_A$.
The relation $\varepsilon_AP_A=0$ implies
$P_A(A)\subseteq I_A$, and therefore also
$P_A(I_A)\subseteq I_A$. Thus $Q_A=P_A|_{I_A}$ is well-defined.

Let $x,y\in I_A$. Since $I_A$ is a two-sided ideal and
$Q_A(x),Q_A(y)\in I_A$, the element $Q_A(x)y+xQ_A(y)+\lambda xy$
belongs to $I_A$. Restricting the Rota-Baxter identity for $P_A$ to
$x,y\in I_A$ gives
$$
Q_A(x)Q_A(y)
=Q_A\bigl(Q_A(x)y+xQ_A(y)+\lambda xy\bigr).
$$
Hence $Q_A$ is a Rota-Baxter operator of weight $\lambda$ on the
nonunital algebra $I_A$.

(\ref{it:b}) For $a,b\in A$, the Rota-Baxter identity gives
$P_A(a)P_A(b)=P_A\bigl(P_A(a)b+aP_A(b)+\lambda ab\bigr)\in P_A(A)$.
Thus $P_A(A)$ is closed under
multiplication. By Item~(\ref{it:a}), it is contained in $I_A$, so it is a
possibly nonunital subalgebra of $I_A$.

(\ref{it:c}) By Item~(\ref{it:a}),
$p_A=P_A(1_A)$ belongs to $I_A$. Substituting
$x=y=1_A$ in the Rota-Baxter identity gives
$$
p_A^2
=P_A\bigl(p_A+p_A+\lambda1_A\bigr)
=2P_A(p_A)+\lambda p_A.
$$
For $a\in A$, the Rota-Baxter identities at $(1_A,a)$ and
$(a,1_A)$ give, respectively,
\begin{align*}
p_AP_A(a)
&=P_A\bigl(p_Aa+P_A(a)+\lambda a\bigr)
 =P_A(p_Aa)+P_A^2(a)+\lambda P_A(a),\\
P_A(a)p_A
&=P_A\bigl(P_A(a)+ap_A+\lambda a\bigr)
 =P_A^2(a)+P_A(ap_A)+\lambda P_A(a).
\end{align*}
Rearranging proves Item~(\ref{it:c}).
\end{proof}

\begin{coro}\label{cor:split}
For an augmented Rota-Baxter algebra $(A,P_A,\varepsilon_A)$, the
sequence of underlying nonunital algebras
$$
0\longrightarrow I_A\longrightarrow A
\mathrel{\mathop{\longrightarrow}^{\varepsilon_A}}\mathbf{k}
\longrightarrow0
$$
is exact and has the distinguished algebra section $u_A$. This section
commutes with the Rota-Baxter operators if and only if $p_A=0$.
\end{coro}

\begin{proof}
Exactness follows from $I_A=\ker\varepsilon_A$ and
$\varepsilon_Au_A=\mathrm{id}_{\mathbf{k}}$. Since
$P_Au_A(c)=cp_A$ for $c\in\mathbf{k}$, the distinguished section
intertwines the zero operator with $P_A$ exactly when $p_A=0$.
\end{proof}

\subsection{Reconstruction from the augmentation ideal}

A Rota-Baxter operator on a nonunital algebra does not usually extend to a Rota-Baxter operator on its ring-theoretic unitization.
General unitization constructions are discussed in
\cite[Sec.~4.4]{G12}.
The following theorem gives the necessary and sufficient
conditions for an augmented extension of a Rota-Baxter operator to the unitization.

For a nonunital algebra $I$, write
$\widetilde I=\mathbf{k}1\oplus I$ for its ordinary algebra
unitization, where $1$ denotes the adjoined unit. Its multiplication
and augmentation are
$$
(c1+x)(d1+y)=cd1+cy+dx+xy,\qquad
\varepsilon(c1+x)=c
$$
for $c,d\in\mathbf{k}$ and $x,y\in I$.

\begin{theorem}\label{thm:kernel-reconstruction}
Let $(I,Q)$ be a nonunital Rota-Baxter algebra of weight $\lambda$ and
let $p\in I$. Define $P:\widetilde I\to\widetilde I$ by
$$
P(c1+x)=cp+Q(x).
$$
Then $(\widetilde I,P,\varepsilon)$ is an augmented Rota-Baxter algebra
of weight $\lambda$ if and only if
$
p^2=2Q(p)+\lambda p
$
and, for every $x\in I$,
$$
pQ(x)=Q\bigl(px+Q(x)+\lambda x\bigr),\qquad
Q(x)p=Q\bigl(Q(x)+xp+\lambda x\bigr).
$$
\end{theorem}

\begin{proof}
Suppose first that $(\widetilde I,P,\varepsilon)$ is an augmented
Rota-Baxter algebra. Since $P(1)=p$ and $P(x)=Q(x)$ for $x\in I$, the
Rota-Baxter identity applied to $1$ and $1$ gives
$$
p^2=P(2p+\lambda1)=2Q(p)+\lambda p.
$$
For $x\in I$, applying the identity to $1$ and $x$ gives
$$
pQ(x)=P\bigl(px+Q(x)+\lambda x\bigr)
=Q\bigl(px+Q(x)+\lambda x\bigr).
$$
Applying it to $x$ and $1$ gives
$$
Q(x)p=P\bigl(Q(x)+xp+\lambda x\bigr)
=Q\bigl(Q(x)+xp+\lambda x\bigr).
$$
This proves the necessity of the three conditions.

Conversely, assume that the three conditions hold. For
$c,d\in\mathbf{k}$ and $x,y\in I$, we have
\begin{align*}
P(c1+x)P(d1+y)
={}&cdp^2+cpQ(y)+dQ(x)p+Q(x)Q(y)\\
={}&cd\bigl(2Q(p)+\lambda p\bigr)
 +cQ\bigl(py+Q(y)+\lambda y\bigr)\\
&+dQ\bigl(Q(x)+xp+\lambda x\bigr)
 +Q\bigl(Q(x)y+xQ(y)+\lambda xy\bigr)\\
={}&P\Bigl(\lambda cd1+2cdp
 +c\bigl(py+Q(y)+\lambda y\bigr)\\
&+d\bigl(Q(x)+xp+\lambda x\bigr)
 +Q(x)y+xQ(y)+\lambda xy\Bigr)\\
={}&P\bigl(P(c1+x)(d1+y)+(c1+x)P(d1+y)
                 +\lambda(c1+x)(d1+y)\bigr).
\end{align*}
Here the second equality uses the three assumed conditions and
the Rota-Baxter identity for $Q$.
Thus $P$ satisfies the Rota-Baxter identity on $\widetilde I$.
Also $P(\widetilde I)\subseteq I=\ker\varepsilon$, so
$\varepsilon P=0$.
\end{proof}

Let $\mathbf{RBTrip}_\lambda$ be the category whose objects are the
triples $(I,Q,p)$ in Theorem \ref{thm:kernel-reconstruction} and whose
morphisms $g:(I,Q,p)\to(J,Q',p')$ are nonunital algebra morphisms
satisfying $gQ=Q'g$ and $g(p)=p'$.

\begin{coro}\label{cor:reconstruction-equivalence}
Restriction to the augmentation ideal and ordinary unitization give an
equivalence
$$
\mathbf{AugRBAlg}_\lambda\simeq\mathbf{RBTrip}_\lambda.
$$
\end{coro}

\begin{proof}
For an augmented Rota-Baxter algebra, use
$A=\mathbf{k}1_A\oplus I_A$, $Q_A=P_A|_{I_A}$ and $p_A=P_A(1_A)$.
Theorem \ref{thm:kernel-reconstruction} shows that restriction and
unitization recover $A$ through the natural isomorphism
$\mathbf{k}1\oplus I_A\to A$, $c1+x\mapsto c1_A+x$.  A morphism restricts to the
augmentation ideals, while a morphism $g:I\to J$ of triples has the
unique augmented unital extension
$\widetilde g(c1+x)=c1+g(x)$.  These extensions and restrictions commute with composition, and
the displayed isomorphisms are natural.
\end{proof}

\begin{remark}
For a commutative augmentation ideal, the two mixed identities in Theorem
\ref{thm:kernel-reconstruction} coincide. In general they record the two
ways in which the element $p=P(1)$ controls the failure of
$P^2=-\lambda P$.
\end{remark}

An augmented extension to the ordinary unitization need not exist.
For example, let $\mathbf{k}$ be a field, let $I=\mathbf{k}x$
with $I^2=0$, and take $Q=\mathrm{id}_I$ at weight zero.
All products in $I$ vanish, so $Q$ satisfies the Rota-Baxter
identity. For every $p\in I$, the mixed identities in
Theorem \ref{thm:kernel-reconstruction} would give $Q^2=0$,
contradicting $Q(x)=x\neq0$.

\subsection{Normalized augmented Rota-Baxter algebras}

\begin{defn}
An augmented Rota-Baxter algebra $(A,P_A,\varepsilon_A)$ is
\textbf{normalized} if $P_A(1_A)=0$.
\end{defn}

For a normalized augmented Rota-Baxter algebra, Proposition
\ref{prop:unit-identities}(\ref{it:c}) implies $P_A^2=-\lambda P_A$. Operators satisfying
this relation are usually called quasi-idempotent~\cite{Ag20}. When $\lambda$ is
invertible, their description in terms of a direct sum of two subalgebras is
standard. See \cite{GL21}. We record the consequences
of normalization in the augmented setting.

\begin{theorem}\label{thm:unit-zero}
Let $(A,P_A,\varepsilon_A)$ be an augmented Rota-Baxter algebra of weight
$\lambda$.
\begin{enumerate}
\item\label{it:uz-a} The following are equivalent:
\begin{enumerate}
\makeatletter
\renewcommand{\p@enumii}{}
\makeatother
\item\label{it:uz-i} $(A,P_A,\varepsilon_A)$ is normalized.
\item\label{it:uz-ii} the unit morphism $u_A:(\mathbf{k},0)\longrightarrow(A,P_A)$
is a morphism of Rota-Baxter algebras.
\end{enumerate}
If these conditions hold, then
$$
P_A^2=-\lambda P_A.
$$

\item\label{it:uz-b} The category of normalized augmented Rota-Baxter algebras
is equivalent to the category of nonunital Rota-Baxter algebras
$(I,Q)$ of weight $\lambda$ satisfying
$$
Q^2=-\lambda Q,
$$
with Rota-Baxter algebra morphisms. The functors are restriction
to the augmentation ideal and ordinary unitization.

\item\label{it:uz-c} Assume that $\lambda$ is invertible in $\mathbf{k}$. Then a normalized augmented Rota-Baxter algebra is equivalently an augmented algebra $A$
whose augmentation ideal decomposes as a direct sum
$$
I_A=I_A^-\oplus I_A^+
$$
into two nonunital subalgebras. Under this correspondence,
$$
P_A(c1_A+x_-+x_+)=-\lambda x_-,
\qquad
x_-\in I_A^-,\quad x_+\in I_A^+.
$$
A morphism between two such objects is an augmented algebra morphism that
preserves both summands.
\end{enumerate}
\end{theorem}

\begin{proof}
(\ref{it:uz-a}) By definition, $(A,P_A,\varepsilon_A)$ is normalized exactly when $P_A(1_A)=0$. The unit morphism intertwines the zero operator on $\mathbf{k}$ with $P_A$ if and only if $P_A(1_A)=0$. This proves the equivalence of (\ref{it:uz-i}) and (\ref{it:uz-ii}). With
$p_A=0$, the identity in Proposition \ref{prop:unit-identities}(\ref{it:c}) reduces
to $P_A^2=-\lambda P_A$.

(\ref{it:uz-b}) Theorem \ref{thm:kernel-reconstruction} with $p=0$ shows that the mixed
conditions are exactly $Q^2=-\lambda Q$. This proves (\ref{it:uz-b}), including its statement about morphisms.

(\ref{it:uz-c}) Since $\lambda$ is invertible and
$P_A^2=-\lambda P_A$, every $x\in I_A$ has the decomposition
$$
x=-\lambda^{-1}P_A(x)+\bigl(x+\lambda^{-1}P_A(x)\bigr)
\in\operatorname{Im}(P_A|_{I_A})+\ker(P_A|_{I_A}).
$$
The intersection is zero: if $x=P_A(y)$ and $P_A(x)=0$, then
$0=P_A^2(y)=-\lambda x$.
The image is closed under multiplication by the Rota-Baxter
identity. If $P_A(x)=P_A(y)=0$, the same identity gives
$\lambda P_A(xy)=0$, so the kernel is also closed under
multiplication. Thus
$$
I_A=\operatorname{Im}(P_A|_{I_A})\oplus\ker(P_A|_{I_A})
$$
is the standard decomposition into subalgebras
\cite[Proposition~2.3]{GL21}.
These are $I_A^-$ and $I_A^+$, respectively.
Conversely, apply the same splitting construction to
$A=I_A^-\oplus(\mathbf{k}1_A\oplus I_A^+)$.
It gives $P_A=-\lambda$ times the projection onto $I_A^-$.
This operator has image in $I_A$ and annihilates $1_A$, so it is
augmented and normalized. Finally, an augmented algebra morphism
commutes with these projections exactly when it preserves both summands.
\end{proof}

At weight zero, normalization requires the restricted operator to have
square zero. For invertible weight, it makes the operator a scalar
multiple of a projection. Both cases will be used in the finite
dimensional examples of Section \ref{subsec:braid}.

\section{Rota-Baxter operators on tensor products}\label{sec:tp}

We construct a Rota-Baxter operator on the ordinary tensor product algebra $A\otimes B$ using the augmentations of the two tensor factors and show that the resulting product defines a monoidal structure. We then study the flip-conjugated construction, prove a uniqueness result within a natural class of linear combinations, and examine symmetry and braidings.

\subsection{Construction and monoidal structure}
The tensor product $P_A\otimes P_B$ of two Rota-Baxter operators need not be a Rota-Baxter operator on $A\otimes B$.

\begin{exam}\label{ex:tensor-obstruction}
Assume that $\mathbf{k}$ is a field of characteristic zero. Let
$A=\mathbf{k}[x]$ and let $P:A\to A$ be the integration operator
$P(f)(x)=\int_0^x f(t)\,dt$. Although $P$ is a Rota-Baxter operator of
weight zero, $P\otimes P$ is not a Rota-Baxter operator on $A\otimes A$.
Now $(P\otimes P)(1_A\otimes1_A)=x\otimes x$. If $P\otimes P$
satisfied the Rota-Baxter identity of weight zero, evaluation at
$(1_A\otimes1_A,1_A\otimes1_A)$ would give
$$
(x\otimes x)^2
=(P\otimes P)\bigl(2x\otimes x\bigr)
=P(2x)\otimes P(x)
=x^2\otimes\frac{x^2}{2},
$$
This contradicts $(x\otimes x)^2=x^2\otimes x^2$.
Thus the tensor product of the two operators does not give
the required weight-zero Rota-Baxter structure.
\end{exam}

We now show that the augmentations provide a replacement for $P_A\otimes P_B$ on the ordinary tensor product algebra. 

\begin{theorem}\label{thm:tensor-product}
Let $(A,P_A,\varepsilon_A)$ and $(B,P_B,\varepsilon_B)$ be
augmented Rota-Baxter algebras of weight $\lambda$ over $\mathbf{k}$.
Equip $A\otimes B$ with its ordinary tensor product multiplication
and unit $1_A\otimes1_B$. The linear operator
$$
P_{A\boxtimes B}=P_A\otimes e_B+\mathrm{id}_A\otimes P_B,
\qquad e_B=u_B\varepsilon_B,
$$
together with the augmentation $\varepsilon_A\otimes\varepsilon_B$,
makes $A\otimes B$ an augmented Rota-Baxter algebra of weight
$\lambda$. In particular, for $a\in A$ and $b\in B$, we have
$$
P_{A\boxtimes B}(a\otimes b)
=P_A(a)\otimes\varepsilon_B(b)1_B+a\otimes P_B(b).
$$
\end{theorem}

\begin{proof}
The map $\varepsilon_A\otimes\varepsilon_B$ is a unital algebra
morphism, and
\begin{align*}
(\varepsilon_A\otimes\varepsilon_B)P_{A\boxtimes B}
&=(\varepsilon_AP_A)\otimes(\varepsilon_Be_B)
  +\varepsilon_A\otimes(\varepsilon_BP_B)\\
&=0\qquad
  (\text{since }\varepsilon_AP_A=0\text{ and }\varepsilon_BP_B=0).
\end{align*}
It remains to verify the Rota-Baxter identity. By bilinearity,
it suffices to evaluate it at $a\otimes b$ and $a'\otimes b'$,
where $a,a'\in A$ and $b,b'\in B$.
Since $e_B$ is multiplicative, has scalar image and satisfies
$e_BP_B=0$, we have
$$
e_B(P_B(b)b')=e_B(bP_B(b'))=0,\qquad
P_B(e_B(b)b')=e_B(b)P_B(b').
$$
We also have $P_B(be_B(b'))=P_B(b)e_B(b')$.
These identities give
\begin{align*}
&P_{A\boxtimes B}
 \Bigl(P_{A\boxtimes B}(a\otimes b)(a'\otimes b')
 +(a\otimes b)P_{A\boxtimes B}(a'\otimes b')
 +\lambda aa'\otimes bb'\Bigr)\\
={}&(P_A\otimes e_B+\mathrm{id}_A\otimes P_B)
 \Bigl(P_A(a)a'\otimes e_B(b)b'
       +aP_A(a')\otimes be_B(b')\\
&\hspace{46mm}
       +aa'\otimes\bigl(P_B(b)b'+bP_B(b')+\lambda bb'\bigr)\Bigr)
 \displaybreak[1]\\
={}&P_A\bigl(P_A(a)a'+aP_A(a')+\lambda aa'\bigr)
       \otimes e_B(b)e_B(b')\\
&+aa'\otimes P_B\bigl(P_B(b)b'+bP_B(b')+\lambda bb'\bigr)\\
&+P_A(a)a'\otimes e_B(b)P_B(b')\\
&+aP_A(a')\otimes P_B(b)e_B(b')
 \qquad(\text{since }e_BP_B=0)\displaybreak[1]\\
={}&P_A(a)P_A(a')\otimes e_B(b)e_B(b')
 +aa'\otimes P_B(b)P_B(b')\\
&+P_A(a)a'\otimes e_B(b)P_B(b')\\
&+aP_A(a')\otimes P_B(b)e_B(b')
 \qquad(\text{by Definition }\ref{def:rb})\displaybreak[1]\\
={}&\bigl(P_A(a)\otimes e_B(b)+a\otimes P_B(b)\bigr)
     \bigl(P_A(a')\otimes e_B(b')+a'\otimes P_B(b')\bigr)\\
={}&P_{A\boxtimes B}(a\otimes b)P_{A\boxtimes B}(a'\otimes b').
\end{align*}
\end{proof}

We write
$$
(A,P_A,\varepsilon_A)\boxtimes(B,P_B,\varepsilon_B)
:=(A\otimes B,P_{A\boxtimes B},\varepsilon_A\otimes\varepsilon_B)
$$
and refer the reader to~
\cite{AM10,K95,M98} for standard terminology and conventions for monoidal categories. 
\begin{theorem}\label{thm:monoidal}
The product $\boxtimes$, acting on morphisms by
$f\boxtimes g=f\otimes g$, defines a bifunctor on
$\mathbf{AugRBAlg}_\lambda$.
Together with the unit object
$\mathbf{1}=(\mathbf{k},0,\mathrm{id}_{\mathbf{k}})$ and the
natural constraints
\begin{align*}
\alpha_{A,B,C}:(A\boxtimes B)\boxtimes C
 &\longrightarrow A\boxtimes(B\boxtimes C),
 &(a\otimes b)\otimes c&\longmapsto a\otimes(b\otimes c),\\
\ell_A:\mathbf{1}\boxtimes A&\longrightarrow A,
 &t\otimes a&\longmapsto ta,\\
\rho_A:A\boxtimes\mathbf{1}&\longrightarrow A,
 &a\otimes t&\longmapsto ta,
\end{align*}
where $a\in A$, $b\in B$, $c\in C$ and $t\in\mathbf{k}$,
it makes $\mathbf{AugRBAlg}_\lambda$ a monoidal category.
\end{theorem}

\begin{proof}
For augmented Rota-Baxter morphisms $f:A\to A'$ and $g:B\to B'$,
we have $ge_B=e_{B'}g$. Thus
\begin{align*}
(f\otimes g)P_{A\boxtimes B}
&=(fP_A)\otimes(ge_B)+f\otimes(gP_B)\\
&=(P_{A'}f)\otimes(e_{B'}g)+f\otimes(P_{B'}g)\\
&=P_{A'\boxtimes B'}(f\otimes g),\\
(\varepsilon_{A'}\otimes\varepsilon_{B'})(f\otimes g)
&=(\varepsilon_{A'}f)\otimes(\varepsilon_{B'}g)
 =\varepsilon_A\otimes\varepsilon_B.
\end{align*}
Since the ordinary tensor product preserves identities and
composition, these maps define the asserted bifunctor.

For $a\in A$, $b\in B$ and $c\in C$, we have
\begin{align*}
&\alpha_{A,B,C}P_{(A\boxtimes B)\boxtimes C}
  \bigl((a\otimes b)\otimes c\bigr)\\
={}&P_A(a)\otimes\bigl(e_B(b)\otimes e_C(c)\bigr)
   +a\otimes\bigl(P_B(b)\otimes e_C(c)\bigr)\\
&+a\otimes\bigl(b\otimes P_C(c)\bigr)\\
={}&P_A(a)\otimes e_{B\otimes C}(b\otimes c)\\
&+a\otimes P_{B\boxtimes C}(b\otimes c)
 \qquad(\text{since }e_{B\otimes C}=e_B\otimes e_C)\\
={}&P_{A\boxtimes(B\boxtimes C)}
  \alpha_{A,B,C}\bigl((a\otimes b)\otimes c\bigr).
\end{align*}
For $t\in\mathbf{k}$ and $a\in A$, the equalities
$P_{\mathbf{1}}=0$ and
$e_{\mathbf{1}}=\mathrm{id}_{\mathbf{k}}$ give
\begin{align*}
\ell_A P_{\mathbf{1}\boxtimes A}(t\otimes a)
&=\ell_A(t\otimes P_A(a))
 =tP_A(a)=P_A\ell_A(t\otimes a),\\
\rho_A P_{A\boxtimes\mathbf{1}}(a\otimes t)
&=\rho_A(P_A(a)\otimes t)
 =tP_A(a)=P_A\rho_A(a\otimes t).
\end{align*}
The associator and unit constraints are unital algebra isomorphisms
and preserve the augmentations, so these calculations make them
augmented Rota-Baxter isomorphisms. Their naturality and the pentagon
and triangle identities hold because the underlying maps are the
ordinary tensor product constraints.
\end{proof}

In particular,
$P_{A\boxtimes B}(1_A\otimes1_B)
=p_A\otimes1_B+1_A\otimes p_B$.
Therefore normalized objects form a full monoidal subcategory.

\begin{coro}\label{thm:iterated}
For augmented Rota-Baxter algebras $(A_i,P_i,\varepsilon_i)$,
$1\leq i\leq n$, the Rota-Baxter operator on the iterated product is
$$
P_{A_1\boxtimes\cdots\boxtimes A_n}
=\sum_{i=1}^{n}
\mathrm{id}_{A_1}\otimes\cdots\otimes\mathrm{id}_{A_{i-1}}
\otimes P_i\otimes e_{A_{i+1}}\otimes\cdots\otimes e_{A_n}.
$$
Empty strings of tensor factors are omitted. The formula is
independent of the parenthesization under the ordinary associators.
\end{coro}

\begin{proof}
The case $n=1$ is immediate. Tensoring a product of $n$ factors with
$A_{n+1}$ replaces its operator by
$P_{A_1\boxtimes\cdots\boxtimes A_n}\otimes e_{A_{n+1}}
+\mathrm{id}_{A_1\otimes\cdots\otimes A_n}\otimes P_{n+1}$.
This proves the formula by induction. Theorem \ref{thm:monoidal}
gives independence of the parenthesization.
\end{proof}

\subsection{Flip conjugation and uniqueness}\label{subsec:flip-conjugation}

Let $\tau_{A,B}:A\otimes B\to B\otimes A$ denote the ordinary flip,
$a\otimes b\mapsto b\otimes a$. Transporting the operator
$P_{B\boxtimes A}$ across the flip gives
$$
P^{\mathrm{fc}}_{A,B}
:=\tau_{B,A}P_{B\boxtimes A}\tau_{A,B}
=P_A\otimes\mathrm{id}_B+e_A\otimes P_B.
$$
We call $P^{\mathrm{fc}}_{A,B}$ the \textbf{flip-conjugated operator} and write
$$
(A,P_A,\varepsilon_A)\boxtimes^{\mathrm{fc}}(B,P_B,\varepsilon_B)
:=(A\otimes B,P^{\mathrm{fc}}_{A,B},\varepsilon_A\otimes\varepsilon_B).
$$

\begin{prop}\label{prop:flip-conjugated}
The product $\boxtimes^{\mathrm{fc}}$ defines a monoidal structure on
$\mathbf{AugRBAlg}_\lambda$ with the ordinary associator and unit
constraints. The ordinary flips give a monoidal equivalence
$$
(\mathbf{AugRBAlg}_\lambda,\boxtimes^{\mathrm{fc}})
\simeq(\mathbf{AugRBAlg}_\lambda,\boxtimes)^{\mathrm{rev}},
$$
where the superscript $\mathrm{rev}$ denotes reversal of the monoidal product.
\end{prop}

\begin{proof}
The flip preserves multiplication and augmentation, and by definition
$\tau_{A,B}P^{\mathrm{fc}}_{A,B}=P_{B\boxtimes A}\tau_{A,B}$.
Transporting the monoidal structure of Theorem \ref{thm:monoidal}
through these isomorphisms gives $\boxtimes^{\mathrm{fc}}$.
The transported associator and unit constraints are the ordinary ones.
\end{proof}

We consider universal formulas with fixed scalar coefficients in the five
terms
$P_A\otimes\mathrm{id}_B$, $P_A\otimes e_B$,
$\mathrm{id}_A\otimes P_B$, $e_A\otimes P_B$ and
$P_A\otimes P_B$. The following theorem shows that only $\boxtimes$ and its
flip-conjugated counterpart occur in this class.

\begin{theorem}\label{thm:tensor-uniqueness}
Assume that $\mathbf{k}$ is a field of characteristic zero, and fix
$a,b,c,d,h\in\mathbf{k}$. For augmented Rota-Baxter algebras
$(A,P_A,\varepsilon_A)$ and $(B,P_B,\varepsilon_B)$ of weight $\lambda$, set
\begin{equation}
\mathscr P_{A,B}:=\mathscr P_{A,B}^{a,b,c,d,h}:=aP_A\otimes\mathrm{id}_B+bP_A\otimes e_B
 +c\mathrm{id}_A\otimes P_B
+d e_A\otimes P_B+hP_A\otimes P_B.
\end{equation}
Define the product on objects by
$$
(A,P_A,\varepsilon_A)\star(B,P_B,\varepsilon_B)
:=(A\otimes B,\mathscr P_{A,B},\varepsilon_A\otimes\varepsilon_B)
$$
and on morphisms by $f\star g=f\otimes g$. If $\star$, together with
the unit object $\mathbf{1}=(\mathbf{k},0,\mathrm{id}_{\mathbf{k}})$ and the
ordinary associator and unit constraints, defines a monoidal structure on
$\mathbf{AugRBAlg}_\lambda$, then exactly one of the following holds:
\begin{align*}
&(a,b,c,d,h)=(0,1,1,0,0),
&&\mathscr P_{A,B}=P_A\otimes e_B+\mathrm{id}_A\otimes P_B,\\
&(a,b,c,d,h)=(1,0,0,1,0),
&&\mathscr P_{A,B}=P_A\otimes\mathrm{id}_B+e_A\otimes P_B.
\end{align*}
Conversely, both formulas define monoidal structures. Thus, within this
five-term class, the only monoidal products are $\boxtimes$ and the
flip-conjugated product $\boxtimes^{\mathrm{fc}}$.
\end{theorem}

\begin{proof}
Since $\star$ is a monoidal product in $\mathbf{AugRBAlg}_\lambda$, each
$\mathscr P_{A,B}$ is a Rota-Baxter operator of weight $\lambda$ and the
ordinary associator and unit constraints are Rota-Baxter morphisms.
The right and left unit constraints give $a+b=1$ and $c+d=1$.
We first use associativity to restrict the coefficients, and then use the
Rota-Baxter identity to exclude the remaining formulas. For the associativity calculation, take
$E=\mathbf{k}1_E\oplus\mathbf{k}x\oplus\mathbf{k}y$ with
$(\mathbf{k}x\oplus\mathbf{k}y)^2=0$ and augmentation
$\varepsilon_E(t1_E+ux+vy)=t$.
Set $P_E(1_E)=0$. For $\lambda=0$, take
$P_E(x)=y$ and $P_E(y)=0$. For $\lambda\neq0$, take
$P_E(x)=-\lambda x$ and $P_E(y)=0$.
The restriction to the square-zero ideal satisfies the Rota-Baxter
identity and $P_E^2=-\lambda P_E$.
Theorem \ref{thm:unit-zero}(\ref{it:uz-b}) therefore applies.
The maps $\mathrm{id}_E,e_E,P_E$ are linearly independent, as
evaluation at $1_E,y,x$ shows. Since $E$ is finite dimensional, the natural map
$\operatorname{End}_{\mathbf{k}}(E)^{\otimes3}
\to\operatorname{End}_{\mathbf{k}}(E^{\otimes3})$ is an isomorphism.
Hence the $27$ triple tensor products of these maps are
linearly independent.

Under the ordinary associator, compare the operators on
$(E\star E)\star E$ and $E\star(E\star E)$. Coefficient comparison gives
\begin{align*}
a^2&=a,& b^2&=b,& ab&=0,\\
c^2&=c,& d^2&=d,& cd&=0,\\
h(a-c)&=0,& h(b-d)&=0.
\end{align*}
The first row follows from the coefficients of
$P_E\otimes\mathrm{id}_E\otimes\mathrm{id}_E$,
$P_E\otimes e_E\otimes e_E$ and
$P_E\otimes e_E\otimes\mathrm{id}_E$.
The second follows from those of
$\mathrm{id}_E\otimes\mathrm{id}_E\otimes P_E$,
$e_E\otimes e_E\otimes P_E$ and
$e_E\otimes\mathrm{id}_E\otimes P_E$.
The last row follows from those of
$P_E\otimes\mathrm{id}_E\otimes P_E$ and $P_E\otimes e_E\otimes P_E$.
Thus each of $(a,b)$ and $(c,d)$ equals $(1,0)$ or $(0,1)$.
When these pairs differ, $h=0$, giving precisely $\boxtimes$ and
$\boxtimes^{\mathrm{fc}}$.

It remains to exclude $(a,b)=(c,d)=(1,0)$ and
$(a,b)=(c,d)=(0,1)$.
For $\lambda=0$, evaluate the Rota--Baxter identity on $E\otimes E$
at $x\otimes1_E$ and $1_E\otimes x$.
For $P_E\otimes\mathrm{id}_E+\mathrm{id}_E\otimes P_E+hP_E\otimes P_E$,
its left and right sides are $y\otimes y$ and $2y\otimes y$.
For $P_E\otimes e_E+e_E\otimes P_E+hP_E\otimes P_E$,
they are $y\otimes y$ and $0$.
Neither identity holds.

For $\lambda\neq0$, take $D=\mathbf{k}[t]/(t^2)$ with its standard
augmentation and $P_D=-\lambda(\mathrm{id}_D-e_D)$.
Write $\bar t$ for the residue class of $t$.
At $\bar t\otimes1_D$ and $1_D\otimes\bar t$, the left side
of the Rota--Baxter identity is $\lambda^2\bar t\otimes\bar t$.
Its right side is
$(2\lambda^2-h\lambda^3)\bar t\otimes\bar t$ when $a=c=1$,
and $-h\lambda^3\bar t\otimes\bar t$ when $a=c=0$.
Thus $h=\lambda^{-1}$ and $h=-\lambda^{-1}$, respectively.
To exclude these values, take $F=\mathbf{k}^{\{0,1,2\}}$ with
pointwise multiplication,
$\varepsilon_F(f)=f(0)$ and
$P_F(f)(n)=\lambda\sum_{i=0}^{n-1}f(i)$.
Write $\delta_i$ for its $i$th coordinate function.
When $a=c=1$ and $h=\lambda^{-1}$, the left side minus the right
side of the Rota--Baxter identity at
$(\delta_0\otimes\delta_1,\delta_1\otimes\delta_0)$ is
$$
\lambda^2(\delta_1+\delta_2)\otimes(\delta_1+\delta_2).
$$
When $a=c=0$ and $h=-\lambda^{-1}$, the difference at
$(\delta_1\otimes\delta_1,\delta_1\otimes\delta_1)$ is
$2\lambda^2\delta_2\otimes\delta_2$.
Both are nonzero in characteristic zero. The converse follows from
Theorem \ref{thm:monoidal} and Proposition \ref{prop:flip-conjugated}.
\end{proof}

The two products differ in the position of the augmentation projection.
In $\boxtimes$, the term involving $P_A$ uses $e_B$, whereas in
$\boxtimes^{\mathrm{fc}}$ the term involving $P_B$ uses $e_A$.
All bialgebra constructions below use $\boxtimes$.

\subsection{The ordinary flip}\label{subsec:flip}

The flip in the preceding subsection was used to transport the operator; it need not itself be a morphism $A\boxtimes B\to B\boxtimes A$. We now determine when it is.

Let $\mathbf{AugAlg}$ denote the category of unital augmented
algebras and unital algebra morphisms preserving augmentations.
There are two direct ways to regard an augmented algebra as an
augmented Rota--Baxter algebra.

\begin{prop}\label{thm:canonical-embedding}
The assignments $(A,\varepsilon_A)\mapsto(A,0,\varepsilon_A)$ and
$(A,\varepsilon_A)\mapsto(A,-\lambda(\mathrm{id}_A-e_A),\varepsilon_A)$ define fully
faithful strong monoidal functors
$\mathbf{AugAlg}\to\mathbf{AugRBAlg}_\lambda$.
Their images, denoted by $\mathcal Z_\lambda$ and
$\mathcal C_\lambda$, are symmetric monoidal subcategories with
the ordinary flip.
\end{prop}

\begin{proof}
The operators were described in Example \ref{ex:basic}
(\ref{it:ex-zero}) and (\ref{it:ex-can}).
Every augmented algebra morphism $f:A\to B$ satisfies $fe_A=e_Bf$,
so both assignments are fully faithful.
For zero operators, preservation of tensor products is immediate.
For the second assignment, we have
\begin{align*}
-\lambda(\mathrm{id}_A-e_A)\otimes e_B-\lambda\mathrm{id}_A\otimes(\mathrm{id}_B-e_B)
&=-\lambda\bigl((\mathrm{id}_A-e_A)\otimes e_B
       +\mathrm{id}_A\otimes(\mathrm{id}_B-e_B)\bigr)\\
&=-\lambda(\mathrm{id}_{A\otimes B}-e_A\otimes e_B).
\end{align*}
Both functors preserve the unit object and the ordinary constraints.
The flips intertwine the corresponding operators and satisfy
the symmetric coherence identities.
At weight zero, the two functors coincide.
\end{proof}

\begin{prop}\label{prop:flip}
Let $(A,P_A,\varepsilon_A)$ and $(B,P_B,\varepsilon_B)$ be objects of the
category $\mathbf{AugRBAlg}_\lambda$.
The ordinary flip $\tau_{A,B}$ is a morphism
$A\boxtimes B\to B\boxtimes A$ in
$\mathbf{AugRBAlg}_\lambda$ if and only if
$$
P_B\otimes(\mathrm{id}_A-e_A)=(\mathrm{id}_B-e_B)\otimes P_A
$$
as maps $B\otimes A\to B\otimes A$.
\end{prop}

\begin{proof}
The flip is an algebra isomorphism and preserves the augmentation. For
$a\in A$ and $b\in B$, the two sides of the condition that the flip
commute with the operators are
\begin{align*}
\tau_{A,B}P_{A\boxtimes B}(a\otimes b)
&=
e_B(b)\otimes P_A(a)+P_B(b)\otimes a,\\
P_{B\boxtimes A}\tau_{A,B}(a\otimes b)
&=
P_B(b)\otimes e_A(a)+b\otimes P_A(a).
\end{align*}
These expressions are equal for all $a,b$ if and only if
$$
P_B(b)\otimes(a-e_A(a))
=
(b-e_B(b))\otimes P_A(a),
$$
which is the stated condition.
\end{proof}

\begin{theorem}\label{thm:maximal-symmetric}
Assume that $\mathbf{k}$ is a field. Every full monoidal
subcategory of $\mathbf{AugRBAlg}_\lambda$ on which the ordinary
flips form a symmetry is contained in $\mathcal Z_\lambda$ or
$\mathcal C_\lambda$. If $\lambda\neq0$, these are the two maximal
full monoidal subcategories with this symmetry. If $\lambda=0$,
they coincide.
\end{theorem}

\begin{proof}
Proposition \ref{thm:canonical-embedding} gives the stated symmetries
on $\mathcal Z_\lambda$ and $\mathcal C_\lambda$.
Let $A$ belong to a full monoidal subcategory $\mathcal D$ with this
symmetry. Proposition \ref{prop:flip} applied to $(A,A)$ gives
$$
P_A\otimes(\mathrm{id}_A-e_A)
=(\mathrm{id}_A-e_A)\otimes P_A.
$$
If $I_A=0$, then $A=\mathbf{k}1_A$ and $P_A=0$, so $A$ belongs
to both subcategories. Otherwise choose $0\neq b\in I_A$ and a
linear functional $\ell:A\to\mathbf{k}$ with $\ell(b)=1$.
Evaluation at $a\otimes b$, followed by $\mathrm{id}_A\otimes\ell$,
gives
$$
P_A(a)=\ell(P_A(b))\bigl(a-\varepsilon_A(a)1_A\bigr),
\qquad a\in A.
$$
In particular, $P_A(1_A)=0$. The Rota--Baxter identity at $(1_A,b)$
therefore yields
$$
0=P_A\bigl(P_A(b)+\lambda b\bigr)
=\ell(P_A(b))\bigl(\ell(P_A(b))+\lambda\bigr)b.
$$
Since $\mathbf{k}$ is a field and $b\neq0$, we have
$\ell(P_A(b))=0$ or $\ell(P_A(b))=-\lambda$.
Thus $A$ belongs to $\mathcal Z_\lambda$ or $\mathcal C_\lambda$.

If $\lambda\neq0$, an object of $\mathcal Z_\lambda$ and an object
of $\mathcal C_\lambda$ with nonzero augmentation ideals cannot
both belong to $\mathcal D$. Indeed, the flip condition would give
$$
-\lambda(\mathrm{id}_B-e_B)\otimes(\mathrm{id}_A-e_A)=0,
$$
contradicting its value on a tensor of nonzero elements of
$I_B$ and $I_A$. Hence $\mathcal D$ is contained in one of the two
subcategories. Each contains objects with a nonzero augmentation
ideal, so neither contains the other when $\lambda\neq0$.
This proves maximality. At weight zero the two subcategories coincide.
\end{proof}

\subsection{The obstruction to braidings}\label{subsec:braid}

To exclude every braiding, we need an obstruction to the existence
of any isomorphism between the two orders of a tensor product.
For finite dimensional normalized objects, the rank of the
operator $P_{A\boxtimes B}$ provides such an obstruction.

\begin{prop}\label{prop:affine-rank}
Let $\mathbf{k}$ be a field, and let $A,B$ be finite dimensional
normalized augmented Rota--Baxter algebras.
Put $d_A=\dim_{\mathbf{k}}A$ and $r_A=\operatorname{rank}P_A$,
and define $d_B,r_B$ similarly. Then
$$
d_{A\boxtimes B}=d_Ad_B,\qquad
r_{A\boxtimes B}=r_A+d_Ar_B.
$$
Equivalently, the maps $\chi_A(z)=d_Az+r_A$ on $\mathbb Q$
satisfy $\chi_{A\boxtimes B}=\chi_A\circ\chi_B$.
Thus $A\mapsto\chi_A$ gives a monoid homomorphism from
isomorphism classes under $\boxtimes$ to the maps
$z\mapsto dz+r$, with integers $d\geq1$ and $r\geq0$,
under composition.
\end{prop}

\begin{proof}
Under $B=\mathbf{k}1_B\oplus I_B$, the operator $P_{A\boxtimes B}$
restricts to $P_A\otimes\mathrm{id}$ on $A\otimes\mathbf{k}1_B$
and to $\mathrm{id}_A\otimes P_B$ on $A\otimes I_B$.
Since $P_B(1_B)=0$, we have $P_B(I_B)=P_B(B)$.
The two images lie in distinct direct summands of $A\otimes B$.
Therefore
$$
\operatorname{Im}P_{A\boxtimes B}
=\bigl(\operatorname{Im}P_A\otimes\mathbf{k}1_B\bigr)
\oplus\bigl(A\otimes\operatorname{Im}P_B\bigr).
$$
Taking dimensions gives the rank identity, and the dimension
identity follows from the ordinary tensor product.
Finally, $\chi_A(\chi_B(z))=d_Ad_Bz+d_Ar_B+r_A$.
\end{proof}

\begin{coro}\label{coro:rank-slope}
Under the hypotheses of Proposition \ref{prop:affine-rank}, an
isomorphism $A\boxtimes B\cong B\boxtimes A$ implies
$$
(d_A-1)r_B=(d_B-1)r_A.
$$
Consequently, in a braided monoidal full subcategory of finite
dimensional normalized objects, the rational number
$r_A/(d_A-1)$ is constant on objects of dimension greater than
one. Conversely, objects with a fixed value of this ratio
are closed under $\boxtimes$ and, together with the unit
object, form a full monoidal subcategory.
\end{coro}

\begin{proof}
An isomorphism intertwines the operators, so their ranks are
equal. Proposition \ref{prop:affine-rank} gives
$r_A+d_Ar_B=r_B+d_Br_A$, which is the stated identity.
A braiding supplies such an isomorphism for every pair.
Conversely, if $r_A=s(d_A-1)$ and $r_B=s(d_B-1)$ for
$s\in\mathbb Q$, then
$$
r_{A\boxtimes B}
=s(d_A-1)+d_As(d_B-1)=s(d_Ad_B-1).
$$
This proves the asserted closure.
\end{proof}

\begin{theorem}\label{thm:not-braided}
Let $\mathbf{k}$ be a field. For every $\lambda\in\mathbf{k}$, the
monoidal category
$$
(\mathbf{AugRBAlg}_\lambda,\boxtimes,\mathbf 1)
$$
admits no braiding. In particular, the obstruction is not confined to the
ordinary flip.
\end{theorem}

\begin{proof}
Let $D=\mathbf{k}[t]/(t^2)$ with its standard augmentation
$\varepsilon_D$, and let $Z=(D,0,\varepsilon_D)$.
Suppose first that $\lambda\neq0$ and take
$C=(D,-\lambda(\mathrm{id}_D-e_D),\varepsilon_D)$.
Both objects are normalized, with
$d_Z=d_C=2$, $r_Z=0$ and $r_C=1$.
Proposition \ref{prop:affine-rank} gives
$$
\operatorname{rank}P_{Z\boxtimes C}=2,
\qquad
\operatorname{rank}P_{C\boxtimes Z}=1.
$$
Hence $Z\boxtimes C$ and $C\boxtimes Z$ are not isomorphic as augmented
Rota--Baxter algebras.

Now let $\lambda=0$. Take
$$
E=\mathbf{k}1_E\oplus\mathbf{k}x\oplus\mathbf{k}y,
\qquad
(\mathbf{k}x\oplus\mathbf{k}y)^2=0,
$$
with augmentation $\varepsilon_E(c1_E+ax+by)=c$, and define
$P_E(1_E)=0$, $P_E(x)=y$ and $P_E(y)=0$.
Since $P_E^2=0$, Theorem \ref{thm:unit-zero}(\ref{it:uz-b})
shows that this is a normalized augmented Rota--Baxter algebra
of weight zero. We have $d_E=3$, $r_E=1$, $d_Z=2$ and $r_Z=0$.
Therefore
$$
\operatorname{rank}P_{E\boxtimes Z}=1,
\qquad
\operatorname{rank}P_{Z\boxtimes E}=2,
$$
so these two objects are again nonisomorphic.

If a braiding existed, its component $c_{A,B}$ would be an isomorphism
$A\boxtimes B\to B\boxtimes A$ for every pair $A,B$. Either of the
preceding pairs gives a contradiction.
\end{proof}

For normalized objects of dimension greater than one, the ratio
$r_A/(d_A-1)$ is zero on $\mathcal Z_\lambda$ and, when
$\lambda\neq0$, is one on $\mathcal C_\lambda$.
Corollary \ref{coro:rank-slope} gives a necessary condition for a
subcategory of finite dimensional normalized objects to admit a
braiding. It does not decide whether a subcategory with a fixed
intermediate ratio admits a braiding.

\section{Comonoid objects and the free Rota--Baxter bimonad}\label{sec:monoidal}

The monoidal product $\boxtimes$ makes the Hochschild cocycle identity a
condition on a coproduct in $\mathbf{AugRBAlg}_\lambda$. We first
identify the corresponding comonoid objects. The free Rota--Baxter
functor then provides the passage from bialgebras to cocycle bialgebras.

\subsection{Comonoid objects and the cocycle identity}\label{sec:mc}

A comonoid object in a monoidal category consists of an object $H$,
a coassociative morphism $\Delta:H\to H\boxtimes H$, and a morphism
from $H$ to the unit object satisfying the two counit identities.
We use the following terminology for the structures that occur here.

\begin{defn}[\cite{ZGG16}]\label{def:cb}
\begin{enumerate}
\item\label{it:cb-a}
An \textbf{operated bialgebra} is a bialgebra
$(H,m,u,\Delta,\varepsilon)$ together with a linear operator $P:H\to H$.
\item\label{it:cb-b}
An operated bialgebra is a \textbf{cocycle bialgebra} if
$$
\Delta P=P\otimes1_H+(\mathrm{id}_H\otimes P)\Delta.
$$
Here $P\otimes1_H$ denotes $h\mapsto P(h)\otimes1_H$.
\item\label{it:cb-c}
A \textbf{Rota--Baxter cocycle bialgebra of weight $\lambda$} is a
cocycle bialgebra whose operator satisfies the Rota--Baxter identity of
weight $\lambda$.
\end{enumerate}
If the underlying bialgebra has an antipode, we use the corresponding
term \textbf{cocycle Hopf algebra}.
A morphism is a bialgebra morphism commuting with the operators.
\end{defn}

The cocycle identity in Definition \ref{def:cb}(\ref{it:cb-b})
implies $\varepsilon P=0$. Indeed, applying
$\varepsilon\otimes\varepsilon$ gives
$$
\varepsilon P
=\varepsilon P+(\varepsilon\otimes\varepsilon P)\Delta
=\varepsilon P+\varepsilon P.
$$
We denote the categories of bialgebras and Rota--Baxter cocycle
bialgebras by $\mathbf{Bialg}$ and $\mathbf{RBCBialg}_\lambda$,
respectively. The underlying algebras in these two categories need not be
commutative.

\begin{theorem}\label{thm:comonoid}
Comonoid objects and their morphisms in
$\mathbf{AugRBAlg}_\lambda$ are precisely Rota--Baxter cocycle
bialgebras of weight $\lambda$ and their morphisms.
\end{theorem}

\begin{proof}
Let a comonoid structure on $(H,P,\varepsilon)$ be given.
Its counit must equal $\varepsilon$, because a morphism from
$(H,P,\varepsilon)$ to
$(\mathbf{k},0,\mathrm{id}_{\mathbf{k}})$ preserves the augmentation.
The coproduct is a unital algebra morphism, so the comonoid axioms
make $(H,m,u,\Delta,\varepsilon)$ a bialgebra.
For $h\in H$, use Sweedler notation
$\Delta(h)=\sum h_{(1)}\otimes h_{(2)}$. The counit identity gives
\begin{align*}
P_{H\boxtimes H}\Delta(h)
&=\sum P(h_{(1)})\otimes\varepsilon(h_{(2)})1_H
  +\sum h_{(1)}\otimes P(h_{(2)})\\
&=P(h)\otimes1_H+(\mathrm{id}_H\otimes P)\Delta(h).
\end{align*}
Thus the condition $\Delta P=P_{H\boxtimes H}\Delta$ is the
cocycle identity.

Conversely, a Rota--Baxter cocycle bialgebra satisfies
$\varepsilon P=0$ by the preceding calculation.
The identity
$(\varepsilon\otimes\varepsilon)\Delta=\varepsilon$
makes $\Delta$ an augmented algebra morphism.
The displayed calculation and the cocycle identity show that
$\Delta$ commutes with the operators.
Hence $\Delta$ and $\varepsilon$ are morphisms in
$\mathbf{AugRBAlg}_\lambda$, and their bialgebra axioms are exactly
the comonoid axioms. In both descriptions the morphisms preserve
the multiplication, unit, coproduct, counit and operator.
\end{proof}

Let $\mathbf{CAlgAugRB}_\lambda$ denote the full subcategory of
commutative augmented Rota--Baxter algebras.
It is closed under $\boxtimes$, so Theorem \ref{thm:comonoid}
identifies its comonoid objects with commutative Rota--Baxter
cocycle bialgebras.

\begin{coro}\label{coro:canonical-cocycle}
For every bialgebra $(H,m,u,\Delta,\varepsilon)$, the operator
$P=-\lambda(\mathrm{id}_H-u\varepsilon)$
makes $H$ a Rota--Baxter cocycle bialgebra of weight $\lambda$.
This construction defines a fully faithful functor
$\mathbf{Bialg}\to\mathbf{RBCBialg}_\lambda$.
\end{coro}

\begin{proof}
The second functor in Proposition \ref{thm:canonical-embedding}
is strong monoidal and hence sends comonoid objects to comonoid
objects. Theorem \ref{thm:comonoid} gives the stated cocycle
structure. Every bialgebra morphism commutes with $u\varepsilon$,
so it commutes with the displayed operators. This proves full
faithfulness.
\end{proof}

\begin{coro}\label{coro:iterated-cocycle}
Let $(H,P,\Delta,\varepsilon)$ be a Rota--Baxter cocycle bialgebra.
Write $\Delta^{[1]}=\mathrm{id}_H$ and let
$\Delta^{[n]}:H\to H^{\otimes n}$ be the iterated coproduct.
For every $n\geq1$ and $h\in H$, we have
$$
\Delta^{[n]}P(h)
=\sum_{i=1}^n
\left((\mathrm{id}_H^{\otimes(i-1)}\otimes P)\Delta^{[i]}(h)\right)
\otimes1_H^{\otimes(n-i)}.
$$
\end{coro}

\begin{proof}
The map $\Delta^{[n]}$ is a morphism to the $n$-fold monoidal
product of $(H,P,\varepsilon)$.
Apply Corollary \ref{thm:iterated} to its target operator.
The counit identities remove the final $n-i$ augmentation
projections in the $i$th term and give the formula.
\end{proof}

\subsection{The free Rota--Baxter bimonad}\label{subsec:bimonad}

We use the standard theory of opmonoidal functors and bimonads
\cite{BLV11,S03}.
An opmonoidal functor
$F:(\mathcal C,\otimes,\mathbf 1)\to
(\mathcal D,\boxtimes,\mathbf J)$ has natural maps
$F(X\otimes Y)\to F(X)\boxtimes F(Y)$ and
$F(\mathbf 1)\to\mathbf J$ compatible with associators and
unit constraints. A bimonad is a monad $(T,\eta,\mu)$ for which
$T$ is opmonoidal and $\eta,\mu$ are opmonoidal natural
transformations.
The \textbf{Eilenberg-Moore category} $\mathcal C^T$ has objects
$(A,\alpha)$ with $\alpha:TA\to A$ satisfying
$\alpha\eta_A=\mathrm{id}_A$ and
$\alpha\mu_A=\alpha T(\alpha)$.
A morphism $f:(A,\alpha)\to(B,\beta)$ satisfies
$f\alpha=\beta T(f)$.
For a bimonad, this category has the induced monoidal structure
\cite{BLV11,S03}.
We suppress the ordinary associators and unit constraints below.

The forgetful functor
$U:\mathbf{AugRBAlg}_\lambda\to\mathbf{AugAlg}$
is strict monoidal by Theorem \ref{thm:monoidal}.
For an augmented algebra $A$, let
$(\mathfrak F_\lambda(A),P_{\mathfrak F(A)},j_A)$ be the free
associative Rota--Baxter algebra of weight $\lambda$ generated by
the algebra $A$. Thus every algebra morphism from $A$ to a
Rota--Baxter algebra extends uniquely to a Rota--Baxter morphism from $\mathfrak F_\lambda(A)$ \cite{G12}.
For an arbitrary algebra $A$, this free object can be constructed
from the free unital Rota--Baxter algebra on the underlying set of
$A$ by imposing linearity, multiplicativity and preservation of
the unit on the generating map. The quotient is taken by the
smallest two-sided ideal stable under the operator that contains
these relations. This construction requires no freeness assumption
on the underlying $\mathbf{k}$-module of $A$.
The augmentation of $A$ extends uniquely to a Rota--Baxter morphism
$\varepsilon_{\mathfrak F(A)}:\mathfrak F_\lambda(A)\to(\mathbf{k},0)$.
If an augmented algebra morphism $f:A\to U(B)$ extends to
a Rota--Baxter morphism $\widetilde f:\mathfrak F_\lambda(A)\to B$,
then $\varepsilon_B\widetilde f$ and
$\varepsilon_{\mathfrak F(A)}$ are Rota--Baxter morphisms to
$(\mathbf{k},0)$ agreeing on $A$. Freeness gives
$\varepsilon_B\widetilde f=\varepsilon_{\mathfrak F(A)}$.
Thus the free extension preserves the augmentation, and these
extensions give the adjunction $\mathfrak F_\lambda\dashv U$.
Also, $j_A$ is injective, since the identity of $A$ extends to
a Rota--Baxter morphism $\mathfrak F_\lambda(A)\to(A,0)$
that is a left inverse of $j_A$.

The monad for free commutative Rota--Baxter algebras was considered
in \cite{ZGK16}.
In general, a left adjoint to a strong monoidal functor has
an opmonoidal structure, and the resulting monad is a bimonad
\cite[Sec.~2.5]{BLV11}.
The next theorem applies this principle to the monoidal product
constructed above. We give the structure maps explicitly and
verify that the comparison equivalence recovers this monoidal product.

\begin{theorem}\label{thm:bimonad}
Let $T_\lambda=U\mathfrak F_\lambda$ be the monad on
$\mathbf{AugAlg}$ induced by $\mathfrak F_\lambda\dashv U$.
\begin{enumerate}
\item\label{it:bm-a}
The unique augmented Rota--Baxter morphisms
$$
\delta_{A,B}:\mathfrak F_\lambda(A\otimes B)
\longrightarrow\mathfrak F_\lambda(A)\boxtimes\mathfrak F_\lambda(B),
\qquad
\delta_{A,B}j_{A\otimes B}=j_A\otimes j_B,
$$
together with $\delta_0:\mathfrak F_\lambda(\mathbf{k})\to\mathbf 1$,
give $T_\lambda$ a bimonad structure.
\item\label{it:bm-b}
The comparison functor gives a monoidal equivalence
$$
\mathbf{AugRBAlg}_\lambda\simeq\mathbf{AugAlg}^{T_\lambda}.
$$
The monoidal product of the corresponding monad algebras has operator
$P_A\otimes e_B+\mathrm{id}_A\otimes P_B$.
\item\label{it:bm-c}
The free functor is opmonoidal. Consequently it sends comonoid
objects and their morphisms in $\mathbf{AugAlg}$ to comonoid
objects and their morphisms in $\mathbf{AugRBAlg}_\lambda$.
\end{enumerate}
\end{theorem}

\begin{proof}
(\ref{it:bm-a})
The map $j_A\otimes j_B$ is an augmented algebra morphism to
the underlying algebra of
$\mathfrak F_\lambda(A)\boxtimes\mathfrak F_\lambda(B)$.
The free property therefore gives $\delta_{A,B}$.
The map $\delta_0$ is the unique Rota--Baxter extension of the
identity of $\mathbf{k}$ to the unit object.
For augmented algebra morphisms $f:A\to A'$ and $g:B\to B'$,
the two composites in the naturality square are Rota--Baxter
morphisms agreeing on $A\otimes B$, so they are equal.
Likewise, the two opmonoidal associativity composites agree on
$A\otimes B\otimes C$, where both equal
$j_A\otimes j_B\otimes j_C$.
The two unit diagrams agree on $A$.
This proves naturality and coherence.

Write $\eta_A=j_A$ for the monad unit and
$\mu_A:T_\lambda^2A\to T_\lambda A$ for the free evaluation
extending the identity of $T_\lambda A$.
We also use $\delta$ for its underlying augmented algebra maps.
The defining equation for $\delta$ gives
$\delta_{A,B}\eta_{A\otimes B}=\eta_A\otimes\eta_B$.
For the multiplication, the required identity is
$$
\delta_{A,B}\mu_{A\otimes B}
=(\mu_A\otimes\mu_B)\delta_{T_\lambda A,T_\lambda B}
T_\lambda(\delta_{A,B}).
$$
Both sides are Rota--Baxter morphisms from
$\mathfrak F_\lambda(T_\lambda(A\otimes B))$ to
$\mathfrak F_\lambda(A)\boxtimes\mathfrak F_\lambda(B)$.
After precomposition with $j_{T_\lambda(A\otimes B)}$, both
are $\delta_{A,B}$. Freeness therefore proves the identity.
The same comparison on generators gives
$\delta_0\eta_{\mathbf{k}}=\mathrm{id}_{\mathbf{k}}$ and
$\delta_0\mu_{\mathbf{k}}=\delta_0T_\lambda(\delta_0)$.
Hence $T_\lambda$ is a bimonad.

(\ref{it:bm-b})
Let $\alpha:T_\lambda A\to A$ be a monad algebra. Thus $\alpha$
is an augmented algebra morphism satisfying
$\alpha j_A=\mathrm{id}_A$ and
$\alpha\mu_A=\alpha T_\lambda(\alpha)$.
Define
$$
P_\alpha(a)=\alpha\bigl(P_{\mathfrak F(A)}(j_A(a))\bigr).
$$
For $z\in T_\lambda A$, apply the second monad algebra identity
to $P_{\mathfrak F(T_\lambda A)}(j_{T_\lambda A}(z))$.
Since $\mu_A$ and $\mathfrak F_\lambda(\alpha)$ are
Rota--Baxter morphisms, this gives
$$
\alpha\bigl(P_{\mathfrak F(A)}(z)\bigr)
=\alpha\bigl(P_{\mathfrak F(A)}(j_A(\alpha(z)))\bigr)
=P_\alpha(\alpha(z)).
$$
Therefore $\alpha P_{\mathfrak F(A)}=P_\alpha\alpha$.
For $a,b\in A$, it follows that
\begin{align*}
P_\alpha(a)P_\alpha(b)
&=\alpha\bigl(P_{\mathfrak F(A)}(j_A(a))
              P_{\mathfrak F(A)}(j_A(b))\bigr)\\
&=\alpha P_{\mathfrak F(A)}
 \bigl(P_{\mathfrak F(A)}(j_A(a))j_A(b)
       +j_A(a)P_{\mathfrak F(A)}(j_A(b))
       +\lambda j_A(a)j_A(b)\bigr)\\
&=P_\alpha\bigl(P_\alpha(a)b+aP_\alpha(b)+\lambda ab\bigr).
\end{align*}
Also $\varepsilon_AP_\alpha=0$, because $\alpha$ preserves
the augmentation.

Conversely, let $(A,P,\varepsilon_A)$ be an augmented Rota--Baxter
algebra. Its free evaluation
$\alpha:\mathfrak F_\lambda(A)\to A$ extends $\mathrm{id}_A$
and commutes with the operators.
The identity $\alpha j_A=\mathrm{id}_A$ holds by construction.
The maps $\alpha\mu_A$ and $\alpha T_\lambda(\alpha)$ are
Rota--Baxter morphisms agreeing on $T_\lambda A$, so they are
equal. Starting from $\alpha$, the recovered operator makes $\alpha$
the Rota--Baxter extension of $\mathrm{id}_A$, so freeness
recovers the original action. Starting from $P$, the formula
$P_\alpha(a)=\alpha(P_{\mathfrak F(A)}(j_A(a)))=P(a)$
recovers the original operator.
A morphism of monad algebras commutes with the recovered operators
by applying its defining equation to
$P_{\mathfrak F(A)}(j_A(a))$.
Conversely, an augmented algebra morphism commuting with the
operators respects evaluations by freeness.
This proves the comparison equivalence.

For monad algebras $(A,\alpha)$ and $(B,\beta)$, the
Eilenberg-Moore tensor action is
$(\alpha\otimes\beta)\delta_{A,B}$.
Applied to $P_{\mathfrak F(A\otimes B)}(j_{A\otimes B}(a\otimes b))$,
it gives
$P_A(a)\otimes\varepsilon_B(b)1_B+a\otimes P_B(b)$.
Thus the comparison equivalence preserves the monoidal product $\boxtimes$
and its constraints.

(\ref{it:bm-c})
The maps constructed in (\ref{it:bm-a}) are opmonoidal maps for
$\mathfrak F_\lambda$ itself. An opmonoidal functor sends a
comonoid $(A,\Delta_A,\varepsilon_A)$ to the comonoid with
coproduct $\delta_{A,A}\mathfrak F_\lambda(\Delta_A)$ and
counit $\delta_0\mathfrak F_\lambda(\varepsilon_A)$.
This gives the assertion.
\end{proof}

\begin{coro}\label{cor:bialgebraic-monadicity}
The bimonad $T_\lambda$ induces a monad
$\widehat T_\lambda$ on $\mathbf{Bialg}$, and there is an
equivalence
$$
\mathbf{RBCBialg}_\lambda\simeq
\mathbf{Bialg}^{\widehat T_\lambda}
$$
commuting with the forgetful functors to $\mathbf{Bialg}$.
\end{coro}

\begin{proof}
Bialgebras are comonoid objects in $\mathbf{AugAlg}$.
The opmonoidal functor $T_\lambda$ preserves comonoids, and the
opmonoidality of its unit and multiplication makes them comonoid
morphisms. This defines $\widehat T_\lambda$.
By Theorem \ref{thm:bimonad}(\ref{it:bm-b}), an algebra over this
monad is an augmented Rota--Baxter algebra whose evaluation is a
bialgebra morphism.
Compatibility of evaluation with the coproduct is the condition
that the coproduct be a morphism of monad algebras to their
Eilenberg--Moore monoidal product.
Theorems \ref{thm:bimonad}(\ref{it:bm-b}) and \ref{thm:comonoid}
identify this condition with the cocycle identity.
The same correspondence preserves morphisms and the underlying
bialgebras.
\end{proof}

\section{Free cocycle bialgebras and Hopf structures}\label{sec:frb}

We now apply the opmonoidal free functor to bialgebras.
The comonoid description gives the coproduct, and freeness
gives its universal property. We then carry out the construction
in the commutative category and prove a common filtration theorem
for both constructions. The final examples describe the resulting
coproducts and antipodes.

\subsection{The free construction on a bialgebra}

Cocycle bialgebra structures on free Rota--Baxter algebras generated by a set were constructed using rooted forests in
\cite{ZGG16} and by angularly decorated rooted trees in \cite{G12}. 
Here the input is a bialgebra $A$ with its
prescribed coproduct, and the free construction extends that
coproduct by its universal property.

\begin{theorem}\label{thm:free-lifting}
Let $(A,m_A,u_A,\Delta_A,\varepsilon_A)$ be a bialgebra.
\begin{enumerate}
\item\label{it:fl-a}
There is a unique augmented Rota--Baxter algebra morphism
$$
\Delta_{\mathfrak F}:\mathfrak F_\lambda(A)
\longrightarrow\mathfrak F_\lambda(A)\boxtimes\mathfrak F_\lambda(A)
$$
such that
$\Delta_{\mathfrak F}j_A=(j_A\otimes j_A)\Delta_A$.
Together with the unique Rota--Baxter morphism
$\varepsilon_{\mathfrak F}:\mathfrak F_\lambda(A)\to(\mathbf{k},0)$
extending $\varepsilon_A$, it makes $\mathfrak F_\lambda(A)$
a Rota--Baxter cocycle bialgebra.
\item\label{it:fl-b}
These structures give a functor
$\mathfrak F_\lambda:\mathbf{Bialg}\to\mathbf{RBCBialg}_\lambda$
left adjoint to the forgetful functor. Thus every bialgebra
morphism $f:A\to H$ to a Rota--Baxter cocycle bialgebra
extends uniquely to a Rota--Baxter cocycle bialgebra morphism
$\widetilde f:\mathfrak F_\lambda(A)\to H$.
\end{enumerate}
\end{theorem}

\begin{proof}
(\ref{it:fl-a})
The bialgebra $A$ is a comonoid in $\mathbf{AugAlg}$.
Theorem \ref{thm:bimonad}(\ref{it:bm-c}) gives a comonoid in
$\mathbf{AugRBAlg}_\lambda$ with structure maps
$$
\Delta_{\mathfrak F}=\delta_{A,A}\mathfrak F_\lambda(\Delta_A),
\qquad
\varepsilon_{\mathfrak F}=\delta_0\mathfrak F_\lambda(\varepsilon_A).
$$
Theorem \ref{thm:comonoid} identifies it with a Rota--Baxter cocycle
bialgebra. The displayed maps have the required restrictions to $A$.
Uniqueness holds among Rota--Baxter morphisms to the stated targets
by the free property of $\mathfrak F_\lambda(A)$.

(\ref{it:fl-b})
Freeness gives a unique Rota--Baxter morphism
$\widetilde f:\mathfrak F_\lambda(A)\to H$ extending $f$.
We first check the counit. The Rota--Baxter morphisms
$\varepsilon_H\widetilde f$ and $\varepsilon_{\mathfrak F}$
to $(\mathbf{k},0)$ both restrict to $\varepsilon_A$ on $A$,
so they are equal by freeness.
Thus $\widetilde f$ preserves the augmentation, and
Theorem \ref{thm:monoidal} makes
$\widetilde f\otimes\widetilde f$ an augmented Rota--Baxter
morphism. Consequently the two maps
$$
\Delta_H\widetilde f,\qquad
(\widetilde f\otimes\widetilde f)\Delta_{\mathfrak F}
:\mathfrak F_\lambda(A)\longrightarrow H\boxtimes H
$$
are Rota--Baxter morphisms. Their restrictions to $A$ are
$\Delta_Hf$ and $(f\otimes f)\Delta_A$, which agree since
$f$ is a bialgebra morphism. Hence the two maps agree everywhere.
Together with the counit equality already proved, this shows
that $\widetilde f$ is a cocycle bialgebra morphism.
Uniqueness also implies compatibility with identities and
composition, and the extension correspondence is natural.
\end{proof}

Recall that an element $x$ of a bialgebra $A$ is
\textbf{primitive} if $\Delta_A(x)=x\otimes1_A+1_A\otimes x$,
and an element $g$ is \textbf{group-like} if
$\Delta_A(g)=g\otimes g$ and $\varepsilon_A(g)=1$
\cite{K95}.
Writing $P$ for the free operator, we have
$$
\Delta_{\mathfrak F}(P(j_A(a)))
=P(j_A(a))\otimes1_{\mathfrak F_\lambda(A)}
 +\sum j_A(a_{(1)})\otimes P(j_A(a_{(2)})).
$$
In particular, a group-like element $g\in A$ gives
$\Delta_{\mathfrak F}(P(j_A(g)))
=P(j_A(g))\otimes1_{\mathfrak F_\lambda(A)}
 +j_A(g)\otimes P(j_A(g))$.

\subsection{The free commutative construction}\label{subsec:comm-free}

Let $\mathbf{CAugAlg}$ be the category of commutative augmented
algebras. For a commutative algebra $A$, denote its free
commutative Rota--Baxter algebra of the fixed weight $\lambda$
by $(RB(A),\diamond,P,j_A)$.
The mixable shuffle construction gives
$$
RB(A)=\bigoplus_{m\geq0}A\otimes A^{\otimes m}
$$
with $j_A(a)=a$ in the summand $m=0$ and
$P(a_0\otimes\cdots\otimes a_m)
=1_A\otimes a_0\otimes\cdots\otimes a_m$
\cite{GK00a}.
For completeness, the product is
$$
(a_0\otimes\alpha)\diamond(b_0\otimes\beta)
=a_0b_0\otimes(\alpha *_\lambda\beta),
$$
where $*_\lambda$ is the mixable shuffle product on the tensor
tails. The empty tensor is its unit, and its recursion is
\begin{align*}
(a\otimes\alpha)*_\lambda(b\otimes\beta)
={}&a\otimes\bigl(\alpha *_\lambda(b\otimes\beta)\bigr)
   +b\otimes\bigl((a\otimes\alpha)*_\lambda\beta\bigr)
+\lambda ab\otimes(\alpha *_\lambda\beta).
\end{align*}
Here $a,b\in A$ and $\alpha,\beta$ are tensor tails, possibly empty.
We use $\bar\otimes$ for the external tensor product between
two copies of $RB(A)$ when it is necessary to distinguish it
from the tensors in this direct sum.

If $A$ is augmented, its augmentation extends uniquely to
$\varepsilon_{\mathrm{RB}}:RB(A)\to(\mathbf{k},0)$.
Explicitly,
$$
\varepsilon_{\mathrm{RB}}(a_0\otimes\cdots\otimes a_m)
=
\begin{cases}
\varepsilon_A(a_0),&m=0,\\
0,&m\geq1.
\end{cases}
$$
Thus $RB:\mathbf{CAugAlg}\to\mathbf{CAlgAugRB}_\lambda$
is left adjoint to the commutative forgetful functor.
This is a separate adjunction from
$\mathfrak F_\lambda\dashv U$.
For commutative augmented algebras $A,B$, the map $j_A\otimes j_B$
extends uniquely to a Rota--Baxter morphism
$$
\delta^{\mathrm c}_{A,B}:RB(A\otimes B)\longrightarrow RB(A)\boxtimes RB(B).
$$
The target is commutative. The comparisons on generators in the
proof of Theorem \ref{thm:bimonad}(\ref{it:bm-a}) therefore apply
within the commutative categories and make $RB$ opmonoidal.

Let $\mathbf{CBialg}$ and $\mathbf{CRBCBialg}_\lambda$ denote the
categories of commutative bialgebras and commutative Rota--Baxter
cocycle bialgebras, respectively.
The following coproduct is the specialization $\kappa=\mu=0$
of \cite[Lemma~4.3 and Theorem~4.7]{ZGQ24}.
Here it is obtained from the opmonoidal functor, which also
gives the stated adjunction.

\begin{coro}\label{cor:free-commutative}
For every commutative bialgebra $A$, there is a unique augmented
Rota--Baxter morphism
$$
\Delta_{\mathrm{RB}}:RB(A)\longrightarrow RB(A)\boxtimes RB(A)
$$
satisfying $\Delta_{\mathrm{RB}}j_A=(j_A\bar\otimes j_A)\Delta_A$.
It makes $RB(A)$ a commutative Rota--Baxter cocycle bialgebra with
counit $\varepsilon_{\mathrm{RB}}$.
This construction gives a functor
$RB:\mathbf{CBialg}\to\mathbf{CRBCBialg}_\lambda$
left adjoint to the forgetful functor.
\end{coro}

\begin{proof}
The opmonoidal functor $RB$ sends the comonoid $A$ to a comonoid
in $\mathbf{CAlgAugRB}_\lambda$.
Theorem \ref{thm:comonoid} identifies its coproduct
$\delta^{\mathrm c}_{A,A}RB(\Delta_A)$ with the required
cocycle coproduct. Uniqueness follows from the free commutative
property. For a bialgebra morphism $f:A\to H$ with commutative
target, the free extension $\widetilde f:RB(A)\to H$ preserves
the coproduct and counit by the comparison on generators in
Theorem \ref{thm:free-lifting}(\ref{it:fl-b}).
All target algebras in that comparison are commutative, so
the argument uses exactly the present free property.
\end{proof}

\subsection{Connected filtered Hopf structures}\label{subsec:filtration}

We construct a connected filtration on each of the free cocycle
bialgebras above by a recursion involving only the generating
map, multiplication, and $P$. We then verify the coproduct
condition by induction on the filtration index.

We use the following convention. A \textbf{connected filtered bialgebra}
is a bialgebra $A$ with an increasing exhaustive filtration
$(F_nA)_{n\geq0}$ such that
$F_0A=\mathbf{k}1_A$ and
\begin{equation}\label{eq:filtered-bialgebra}
F_pA\,F_qA\subseteq F_{p+q}A,\qquad
\Delta_A(F_nA)\subseteq\sum_{p+q\leq n}F_pA\otimes F_qA.
\end{equation}
Products of submodules mean the $\mathbf{k}$-spans of products
of their elements. Over a general base ring, tensor products
of filtered submodules in these inclusions mean their images
in the ambient tensor product.

This convention is compatible with the connected cofiltered
coalgebras of \cite[Definition~2.8]{GG19}. Indeed, since the filtration
is increasing, the coproduct condition in \eqref{eq:filtered-bialgebra}
is equivalent to
$\Delta_A(F_nA)\subseteq\sum_{p+q=n}F_pA\otimes F_qA$.
Moreover, because $\varepsilon_A(1_A)=1$, each filtered piece splits as
\begin{equation}\label{eq:filtered-splitting}
F_nA=\mathbf{k}1_A\oplus(F_nA\cap\ker\varepsilon_A).
\end{equation}
Thus the underlying coaugmented coalgebra is connected and cofiltered
in the sense of \cite{GG19}. The multiplicative condition in
\eqref{eq:filtered-bialgebra} is additional and is used below to
construct the filtration on the free Rota--Baxter algebra.

\begin{theorem}\label{thm:filtered-hopf}
Let $A$ be a connected filtered bialgebra.
Let $H=\mathfrak F_\lambda(A)$, or let $H=RB(A)$ when $A$
is commutative, with the cocycle bialgebra structures constructed
above.
\begin{enumerate}
\item\label{it:fh-a}
Define submodules of $H$ recursively by
\begin{align*}
G_0H&=\mathbf{k}1_H,\\
G_nH&=G_{n-1}H+j_A(F_nA)+P(G_{n-1}H)
       +\sum_{i=1}^{n-1}G_iH\,G_{n-i}H\qquad(n\geq1).
\end{align*}
The sum is zero when $n=1$.
Then $(G_nH)_{n\geq0}$ is a connected bialgebra filtration,
with
$$
j_A(F_nA)\subseteq G_nH,\qquad
P(G_nH)\subseteq G_{n+1}H.
$$
Consequently $H$ is a Rota--Baxter cocycle Hopf algebra.
The construction is functorial for filtered bialgebra morphisms.
\item\label{it:fh-b}
For $H=RB(A)$, this filtration has the explicit form
$$
G_NRB(A)=
\sum_{\substack{m\geq0,\ n_0,\ldots,n_m\geq0\\
n_0+\cdots+n_m+m\leq N}}
\operatorname{Im}\bigl(
F_{n_0}A\otimes\cdots\otimes F_{n_m}A
\longrightarrow A^{\otimes(m+1)}\bigr).
$$
Each image is viewed in the corresponding direct summand of $RB(A)$.
\end{enumerate}
\end{theorem}

\begin{proof}
(\ref{it:fh-a})
Every index of $G$ on the right-hand side of the recursion is
less than $n$, so the submodules are well defined.
They are increasing, contain $j_A(F_nA)$ at level $n$, and
satisfy $P(G_nH)\subseteq G_{n+1}H$.
For $p,q\geq1$, the product $G_pH\,G_qH$ is a summand in
the definition of $G_{p+q}H$; if either index is zero, the
same inclusion follows from $G_0H=\mathbf{k}1_H$.
Thus the filtration is multiplicative.
Its union is a Rota--Baxter subalgebra of $H$ containing
$j_A(A)$, since the filtration of $A$ is exhaustive.
As $j_A(A)$ generates $H$ as a Rota--Baxter algebra, this union
is all of $H$.

We prove the coproduct estimate
$$
\Delta_H(G_nH)\subseteq\sum_{p+q\leq n}G_pH\otimes G_qH
$$
by induction on $n$. For $n=0$, it follows from
$\Delta_H(1_H)=1_H\otimes1_H$.
Suppose $n\geq1$ and the estimate holds at all smaller indices.
It holds on the summand $G_{n-1}H$ by induction.
On the generator summand, we have
$$
\Delta_H(j_A(F_nA))
=(j_A\otimes j_A)\Delta_A(F_nA)
\subseteq\sum_{p+q\leq n}G_pH\otimes G_qH.
$$
For $x\in G_{n-1}H$, the cocycle identity and induction give
\begin{align*}
\Delta_H(P(x))
&=P(x)\otimes1_H+(\mathrm{id}_H\otimes P)\Delta_H(x)\\
&\in G_nH\otimes G_0H
 +\sum_{p+q\leq n-1}G_pH\otimes G_{q+1}H\\
&\subseteq\sum_{p+q\leq n}G_pH\otimes G_qH.
\end{align*}
Finally, for $1\leq i\leq n-1$, multiplicativity of $\Delta_H$
and the induction hypothesis yield
\begin{align*}
\Delta_H(G_iH\,G_{n-i}H)
&\subseteq
\left(\sum_{p+q\leq i}G_pH\otimes G_qH\right)
\left(\sum_{r+s\leq n-i}G_rH\otimes G_sH\right)\\
&\subseteq\sum_{p+q+r+s\leq n}G_{p+r}H\otimes G_{q+s}H
 \subseteq\sum_{u+v\leq n}G_uH\otimes G_vH.
\end{align*}
This covers every summand in the recursion and completes the
induction. Since $G_0H=\mathbf{k}1_H$, the filtration is connected.
The counit gives, for every $n$,
\begin{equation}\label{eq:G-splitting}
G_nH=\mathbf{k}1_H\oplus(G_nH\cap\ker\varepsilon_H),
\end{equation}
and, since $(G_nH)$ is increasing, the displayed coproduct estimate
is equivalent to
$$
\Delta_H(G_nH)\subseteq\sum_{p+q=n}G_pH\otimes G_qH.
$$
Hence $(H,\Delta_H,\varepsilon_H,u_H)$ is a connected coaugmented
cofiltered coalgebra in the precise sense of
\cite[Definition~2.8]{GG19}. Thus the hypotheses of \cite[Theorem~3.4]{GG19} are satisfied, and that theorem
gives an antipode over $\mathbf{k}$. In the commutative case, the
filtration in (\ref{it:fh-b}) is the one used in
\cite[Theorem~4.11]{ZGQ24}.

For completeness, the recursive antipode formula follows from the
same cofiltration. If $x\in G_nH\cap\ker\varepsilon_H$, then
\cite[Lemma~2.10 and Remark~3.5]{GG19} give a reduced coproduct
presentation
$$
\Delta_H(x)-x\otimes1_H-1_H\otimes x=\sum x'\otimes x'',
$$
with $x',x''\in G_{n-1}H\cap\ker\varepsilon_H$ (for $n\geq1$). Hence the recursion
strictly lowers the filtration level. Consequently
$$
S(1_H)=1_H,\qquad
S(x)=-x-\sum S(x')x''=-x-\sum x'S(x'')
$$
computes $S$ recursively on the filtration.

For a filtered bialgebra morphism $f:A\to B$, its free
extension preserves products and $P$, and sends $j_A(F_nA)$
into $j_B(F_nB)$. Induction in the defining recursion therefore
shows that this extension preserves every filtration level.
It is a bialgebra morphism by Theorem
\ref{thm:free-lifting}(\ref{it:fl-b}) or Corollary
\ref{cor:free-commutative}, and consequently commutes with
antipodes.

(\ref{it:fh-b})
For $a_i\in F_{n_i}A$, the pure tensor in the summand
$A^{\otimes(m+1)}$ satisfies
$$
a_0\otimes\cdots\otimes a_m
=j_A(a_0)\diamond P\bigl(j_A(a_1)\diamond
P\bigl(\cdots P(j_A(a_m))\cdots\bigr)\bigr).
$$
For $m=0$, this means $j_A(a_0)$.
Applying the multiplication and $P$ inclusions in
(\ref{it:fh-a}) to this expression places the tensor in
$G_{n_0+\cdots+n_m+m}RB(A)$.
Hence the right-hand side of the claimed formula is contained
in $G_NRB(A)$.

Conversely, the submodules on the right-hand side form an
increasing filtration starting at $\mathbf{k}1_H$ and
containing $j_A(F_NA)$ at level $N$.
The operator $P$ prepends $1_A\in F_0A$ to a tensor, so it
raises the filtration index by one.
In a mixable shuffle product, a shuffle adds the two tail
lengths, while each merge reduces the resulting tail length
by one and multiplies two entries in $A$.
The filtration index of a product of entries is at most the
sum of their original indices. Thus the sum of the entry
indices plus the tail length is at most the sum of the two
input bounds. This proves multiplicativity of the displayed
filtration.
Induction in the recursion of (\ref{it:fh-a}) now gives the
reverse inclusion, proving the formula.
\end{proof}

\subsection{Iterates and examples}\label{subsec:examples}

We first recall the classical multiplication formulas for iterates
of the unit.  Put
$$
a_n:=P^n(1_H),\qquad p:=a_1=P(1_H).
$$
In every Rota--Baxter algebra of weight $\lambda$, the Rota--Baxter
identity gives the $\lambda$-divided-power recurrence
\begin{equation}\label{eq:lambda-divided-recurrence}
p\,a_n=(n+1)a_{n+1}+n\lambda a_n\qquad(n\geq0).
\end{equation}
This is the multiplication underlying the $\lambda$-divided power
algebra of \cite{AGKO03}.  Its connection with Stirling numbers was
made explicit in \cite[Theorem~2.1]{Guo05}: if $s(n,k)$ and $S(n,k)$
denote the Stirling numbers of the first and second kinds, respectively,
then for $n\geq1$,
\begin{align}
 p^n&=\sum_{k=1}^n k!\,S(n,k)\lambda^{n-k}a_k,
 \label{eq:stirling-second}\\
 n!\,a_n&=\sum_{k=1}^n s(n,k)\lambda^{n-k}p^k
 =\prod_{j=0}^{n-1}(p-j\lambda1_H).
 \label{eq:stirling-first}
\end{align}
These formulas use only the Rota--Baxter identity and therefore
remain valid in the present, not necessarily commutative, setting.
What is new here is the coalgebraic behavior forced by the cocycle
identity and by $\varepsilon P=0$: the same iterates form a divided-power
coalgebra sequence, and in the Hopf case their antipodes admit an
explicit closed formula.

\begin{prop}\label{thm:unit-iterates}
Let $(H,P,\Delta,\varepsilon)$ be a Rota--Baxter cocycle bialgebra
of weight $\lambda$, let $P^0=\mathrm{id}_H$, and put
$a_n=P^n(1_H)$.
\begin{enumerate}
\item\label{it:ui-a}
For every $n\geq0$,
$$
\Delta(a_n)=\sum_{i=0}^{n}a_{n-i}\otimes a_i.
$$
In particular, $a_1=P(1_H)$ is primitive.
\item\label{it:ui-b}
If $\mathbf{k}$ is a field of characteristic zero and
$P(1_H)\neq0$, then the sub-bialgebra $\mathbf{k}[P(1_H)]$ is a
polynomial Hopf algebra with basis $(a_n)_{n\geq0}$, and its antipode is
$$
S(a_n)=\frac{(-1)^n}{n!}
\prod_{j=0}^{n-1}\bigl(P(1_H)+j\lambda1_H\bigr).
$$
If $P(1_H)=0$, then $a_n=0$ for all $n\geq1$.
\end{enumerate}
\end{prop}

\begin{proof}
(\ref{it:ui-a}) The formula holds for $n=0$.
If it holds at $n$, then the cocycle identity gives
\begin{align*}
\Delta(a_{n+1})
&=a_{n+1}\otimes1_H
 +(\mathrm{id}_H\otimes P)\sum_{i=0}^{n}a_{n-i}\otimes a_i\\
&=\sum_{i=0}^{n+1}a_{n+1-i}\otimes a_i.
\end{align*}
The case $n=1$ says that $a_1=P(1_H)$ is primitive.

(\ref{it:ui-b}) Suppose first that $p=P(1_H)\neq0$.
By Eq.~\eqref{eq:stirling-first}, each $a_n$ is a degree-$n$
polynomial in $p$ with leading coefficient $1/n!$.
It remains to show that $p$ is transcendental over $\mathbf{k}$.
Suppose that $p$ is algebraic, and choose a monic polynomial
$f$ of least degree $d$ with $f(p)=0$.
Applying the counit gives $f(0)=0$. Thus $d=1$ would imply
$p=0$, a contradiction.
If $d>1$, the elements $1_H,p,\ldots,p^{d-1}$ are linearly independent.
Since $p$ is primitive, expanding
$$
0=\Delta(f(p))-f(p)\otimes1_H-1_H\otimes f(p)
$$
in the tensor products of these independent elements, the
coefficient of $p^{d-1}\otimes p$ is $d$.
Terms of lower degree in $f$ cannot contribute to this coefficient,
contradicting characteristic zero. Hence $\mathbf{k}[p]$ is a
polynomial algebra. It is a sub-bialgebra because $p$ is primitive.  The polynomial bialgebra
$\mathbf{k}[p]$ has the canonical antipode determined by $S(p)=-p$; if the
ambient bialgebra $H$ is itself Hopf, this antipode is the restriction of the
antipode of $H$. Equation~\eqref{eq:stirling-first} is triangular with
leading coefficient $1/n!$, so $(a_n)_{n\geq0}$ is a basis of
$\mathbf{k}[p]$. Applying $S$ to Eq.~\eqref{eq:stirling-first}
gives the stated antipode formula.

If $p=0$, then Eq.~\eqref{eq:lambda-divided-recurrence}, starting
at $n=0$, or simply linearity of $P$, gives $a_n=0$ for all $n\geq1$.
\end{proof}

The comparison with the earlier $\lambda$-divided power Hopf algebra
is worth making explicit.  The multiplication formulas
\eqref{eq:lambda-divided-recurrence}--\eqref{eq:stirling-first}
agree with the $\lambda$-divided power algebra studied in
\cite{AGKO03,Guo05}, but the coalgebra structures differ when
$\lambda\neq0$. For the standard basis $(a_n)$ of the Hopf algebra
in \cite{AGKO03}, one has
$\varepsilon(a_1)=\lambda$ and
$$
\Delta(a_1)=a_1\otimes a_0+a_0\otimes a_1
             -\lambda a_0\otimes a_0.
$$
In the present cocycle bialgebra, by contrast,
$\varepsilon(P(1_H))=0$ and $P(1_H)$ is primitive. Thus the classical
Stirling-number multiplication identities coexist here with a
different, cocycle-determined coproduct, which is the source of the
new Hopf formulas above.

\begin{coro}\label{coro:element-iterates}
Let $(H,P,\Delta,\varepsilon)$ be a Rota--Baxter cocycle bialgebra.
\begin{enumerate}
\item\label{it:ei-a}
If $x\in H$ is primitive, then, for $n\geq0$,
$$
\Delta(P^n(x))
=\sum_{i=0}^n P^{n-i}(x)\otimes P^i(1_H)+1_H\otimes P^n(x).
$$
\item\label{it:ei-b}
If $g\in H$ is group-like, then, for $n\geq1$,
$$
\Delta(P^n(g))
=\sum_{i=0}^{n-1}P^{n-i}(g)\otimes P^i(1_H)+g\otimes P^n(g).
$$
\end{enumerate}
\end{coro}

\begin{proof}
(\ref{it:ei-a}) The case $n=0$ is the primitive identity.
If the formula holds at $n$, then
\begin{align*}
\Delta(P^{n+1}(x))
&=P^{n+1}(x)\otimes1_H+(\mathrm{id}_H\otimes P)\Delta(P^n(x))\\
&=P^{n+1}(x)\otimes1_H
 +\sum_{i=0}^n P^{n-i}(x)\otimes P^{i+1}(1_H)
 +1_H\otimes P^{n+1}(x),
\end{align*}
which is the formula at $n+1$.

(\ref{it:ei-b}) The case $n=1$ follows from the cocycle identity
and $\Delta(g)=g\otimes g$.
For the induction step, the same identity gives
\begin{align*}
\Delta(P^{n+1}(g))
&=P^{n+1}(g)\otimes1_H+(\mathrm{id}_H\otimes P)\Delta(P^n(g))\\
&=P^{n+1}(g)\otimes1_H
 +\sum_{i=0}^{n-1}P^{n-i}(g)\otimes P^{i+1}(1_H)
 +g\otimes P^{n+1}(g).
\end{align*}
This proves the formula at $n+1$.
\end{proof}

\begin{exam}\label{ex:noncomm}
Let $\mathbf{k}$ be a field and let
$A=\mathbf{k}\langle x,y\rangle$ be the free associative algebra
with primitive generators $x,y$ and augmentation
$\varepsilon_A(x)=\varepsilon_A(y)=0$.
The filtration by words of length at most $n$ is a connected
bialgebra filtration.
By Theorems \ref{thm:free-lifting} and \ref{thm:filtered-hopf},
$\mathfrak F_\lambda(A)$ is a cocycle Hopf algebra.
We identify $A$ with its image under $j_A$, write $P$ for the
free operator, and write $1$ for the unit of this free algebra.
Since
$\Delta_A(xy)=xy\otimes1+x\otimes y+y\otimes x+1\otimes xy$,
we obtain
\begin{align*}
\Delta_{\mathfrak F}(P(xy))
={}&P(xy)\otimes1+xy\otimes P(1)
 +x\otimes P(y)+y\otimes P(x)+1\otimes P(xy).
\end{align*}
The antipode recursion, together with
$S(x)=-x$, $S(y)=-y$ and $S(xy)=yx$, gives
\begin{align*}
S(P(xy))
&=-P(xy)-S(xy)P(1)-S(x)P(y)-S(y)P(x)\\
&=-P(xy)-yxP(1)+xP(y)+yP(x).
\end{align*}
The order $yx$ in the second term is forced by the antipode.
Also, $xy-yx$ is primitive, so
$$
\Delta_{\mathfrak F}(P(xy-yx))
=P(xy-yx)\otimes1+(xy-yx)\otimes P(1)+1\otimes P(xy-yx).
$$
These calculations use the given noncommutative multiplication
and the prescribed primitive coproduct on the input algebra.
\end{exam}

\begin{exam}\label{ex:comm}
Let $A=\mathbf{k}[x]$ with primitive generator $x$ and
$\varepsilon_A(x)=0$, where $\mathbf{k}$ is a field.
In $RB(A)$, identify $x$ with $j_A(x)$ and write $1$ for
the unit of $RB(A)$. The free operator satisfies
$P(1)=1_A\otimes1_A$,
$P(x)=1_A\otimes x$, and
$P^2(1)=1_A\otimes1_A\otimes1_A$.
The cocycle identity gives
\begin{align*}
\Delta_{\mathrm{RB}}(P(x))
&=P(x)\bar\otimes1+x\bar\otimes P(1)+1\bar\otimes P(x),\\
\Delta_{\mathrm{RB}}(P^2(1))
&=P^2(1)\bar\otimes1+P(1)\bar\otimes P(1)+1\bar\otimes P^2(1).
\end{align*}
Theorem \ref{thm:filtered-hopf} supplies the antipode, with
$S(P(x))=-P(x)+x\diamond P(1)$.
The Rota--Baxter identity gives
$P(1)\diamond P(1)=2P^2(1)+\lambda P(1)$, and hence
$S(P^2(1))=P^2(1)+\lambda P(1)$.
Moreover,
$$
\Delta_{\mathrm{RB}}(P(x))
-\tau_{RB(A),RB(A)}\Delta_{\mathrm{RB}}(P(x))
=x\bar\otimes P(1)-P(1)\bar\otimes x\neq0.
$$
Indeed, $x$ and $P(1)$ are nonzero and lie in distinct tensor
length summands of $RB(A)$. Thus the free cocycle bialgebra
need not be cocommutative, even when its generating bialgebra is.
Since $P(1)\neq0$, this free operator differs from the
operator $-\lambda(\mathrm{id}_{RB(A)}-u_{RB(A)}\varepsilon_{\mathrm{RB}})$
of Corollary \ref{coro:canonical-cocycle}.

For a group-like input, take
$A=\mathbf{k}[z,z^{-1}]$ with
$\Delta_A(z)=z\otimes z$ and $\varepsilon_A(z)=1$.
Corollary \ref{cor:free-commutative} gives
$$
\Delta_{\mathrm{RB}}(P(j_A(z)))
=P(j_A(z))\bar\otimes1_{RB(A)}
 +j_A(z)\bar\otimes P(j_A(z)).
$$
Corollary \ref{coro:element-iterates}(\ref{it:ei-b}) gives the
coproducts of its higher iterates.
\end{exam}

\medskip
\noindent
\textbf{Acknowledgments.} This work was supported by the National
Natural Science Foundation of China (12671029, 12101316 and 12301025).

\noindent
\textbf{Declaration of interests.} The authors declare no conflict of interest.

\smallskip

\noindent
{\bf Data availability. } Data sharing is not applicable as no new data were created or analyzed.

\end{document}